\documentclass[11pt]{amsart}

\usepackage{float}
\usepackage{bbm}
\usepackage{amsfonts}
\usepackage{amsmath}
\usepackage{amssymb}
\usepackage{graphicx}
\usepackage[bindingoffset=0.2in,%
left=1.1in,right=1.1in,top=1.2in,bottom=1.375in,%
footskip=.25in]{geometry}

\usepackage{mathtools}
\usepackage{amsthm}
\usepackage{latexsym}
\usepackage{fancyhdr}
\usepackage{array}
\usepackage{amscd}
\usepackage{lscape}
\usepackage{tikz}
\usepackage{tikz-cd}
\usepackage{float}
\usepackage{bm}
\usepackage{subfigure}
\usepackage{grffile}
\usepackage{array}
\usepackage{longtable}
\usepackage{booktabs}
\usepackage{adjustbox}
\usepackage{diagbox}
\usepackage{algorithm}
 \usepackage{algpseudocode}
 \usepackage{algorithmicx}
 \algdef{SE}[DOWHILE]{Do}{doWhile}{\algorithmicdo}[1]{\algorithmicwhile\ #1}%

\algrenewcommand\algorithmicrequire{\textbf{Input:}}
\algrenewcommand\algorithmicensure{\textbf{Output:}}

\usepackage{ dsfont }
\usepackage{lipsum}
\newcolumntype{C}[1]{>{\centering\arraybackslash}p{#1}}

\definecolor{navy}{HTML}{2F729C} 
\usepackage[hyperfootnotes=false, colorlinks, linkcolor={blue}, citecolor={magenta}, filecolor={blue}, urlcolor={blue}, plainpages=false, pdfpagelabels]{hyperref}
\usepackage[tableposition=above]{caption}
\newcommand{\Q}{{\mathbb Q}}
\newcommand{\Z}{{\mathbb Z}}

\newtheorem{theorem}{Theorem}[section]
\newtheorem{lemma}[theorem]{Lemma}
\newtheorem{proposition}[theorem]{Proposition}
\newtheorem{corollary}[theorem]{Corollary}
\theoremstyle{definition}

\newtheorem{example}[theorem]{Example}

\newtheorem{mainthms}{Theorem}

\newtheorem{remark}[theorem]{Remark}

\numberwithin{equation}{section}

\begin{document}

\title[Explicit classification of isogeny graphs of rational elliptic curves]{Explicit classification of isogeny graphs \\ of rational elliptic curves}

\author{Alexander J. Barrios}
\address{Department of Mathematics, University of St. Thomas, St. Paul, Minnesota 55105}
\email{abarrios@stthomas.edu}


\begin{abstract}
Let $n>1$ be an integer such that $X_{0}\!\left(  n\right)  $ has genus $0$, and let $K$ be a field of characteristic $0$ or relatively prime to $6n$. In this article, we explicitly classify the isogeny graphs of all rational elliptic curves that admit a non-trivial isogeny over $\mathbb{Q}$. We achieve this by introducing $56$ parameterized families of elliptic curves $\mathcal{C}_{n,i}(t,d)$ defined over $K(t,d)$, which have the following two properties for a fixed $n$: the elliptic curves $\mathcal{C}_{n,i}(t,d)$ are isogenous over $K(t,d)$, and there are integers $k_{1}$ and $k_{2}$ such that the $j$-invariants of $\mathcal{C}_{n,k_{1}}(t,d)$ and $\mathcal{C}_{n,k_{2}}(t,d)$ are given by the Fricke parameterizations. As a consequence, we show that if~$E$ is an elliptic curve over a number field $K$ with isogeny class degree divisible by $n\in\left\{4,6,9\right\}  $, then there is a quadratic twist of $E$ that is semistable at all primes $\mathfrak{p}$ of $K$ such that $\mathfrak{p}\nmid n$.
\end{abstract}
\maketitle

\section{Introduction}

Let $K$ be a field and suppose $E_1$ and $E_2$ are isogenous
non-isomorphic elliptic curves defined over~$K$. Then there is an isogeny
$\pi:E_1\rightarrow E_2$ defined over $K$ such that $\ker\pi\cong\mathbb{Z}/n\mathbb{Z}$ for some integer $n>1$. In particular, the $\overline{K}$-isomorphism class of the pair $\left(  E_1,\ker\pi\right)  $
is a non-cuspidal $K$-rational point of $X_{0}\!\left(  n\right)  $. Suppose further that $X_{0}\!\left(  n\right)  $ has genus $0$ (i.e., $n=2,3,\ldots
,10,12,13,16,18,25$). Then the Fricke parameterizations parameterize
the $j$-invariants of the $K$-rational points of $X_{0}\!\left(  n\right)$ \cite[Chapter~2.8]{MR3838339}, \cite[Table~3]{MR3084348}. Specifically, there is
a $t\in K$ such that the $j$-invariants of $E_1$ and $E_2\cong E_1/\ker\pi$
are $j_{n,1}\!\left(  t\right)  $ and $j_{n,2}\!\left(  t\right)  $ (see Table
\ref{ta:jinv}), respectively. 

In this article, we introduce $56$ parameterized
families of elliptic curves $\mathcal{C}_{n,i}\!\left(  t,d\right)$ (see Table~\ref{ta:curves}) defined
over $K\!\left(  t,d\right)  $ for $K$ a field of characteristic $0$ or relatively prime to $6n$.
We note that for each $n$, there are at most $8$ elliptic curves in the family $\{\mathcal{C}_{n,i}\!\left(  t,d\right)\}_i$. Thus, $1\leq i \leq 8$.
These families have the property that the elliptic curves
$\mathcal{C}_{n,i}\!\left(  t,d\right)  $ are non-isomorphic isogenous
elliptic curves over $K\!\left(  t,d\right)  $ (see Proposition \ref{thmcalC}). 
The notation is chosen so that
$\mathcal{C}_{n,i}\!\left(  t,d\right)  $ is the quadratic twist of
$\mathcal{C}_{n,i}\!\left(  t,1\right)  $ by $d$. 
Furthermore, the isogeny class of $\mathcal{C}_{n,1}\!\left(  t,d\right)  $ contains $\{
[  \mathcal{C}_{n,i}\!\left(  t,d\right)  ]  _{K(t,d)}\}_i  $ where
$\left[  E\right]  _{F}$ denotes the $F$-isomorphism class of~$E$. A partial isogeny graph of $\mathcal{C}_{n,1}\!\left(  t,d\right)  $ is given in Table~\ref{isographs}. 
In fact, the indexing has been chosen to align with the labeling of the isogeny graphs given in \cite{MR4203041}.
These parameterized families
have the property that for a given $n$, there are $k_{1}$ and $k_{2}$ such
that the $j$-invariant of $\mathcal{C}_{n,k_{i}}\!\left(  t,d\right)  $ is
$j_{n,i}\!\left(  t\right)  $. The values of $k_{i}$ are given below:%
\[%
\begin{array}
[c]{ccccccccc}%
n & 2,3,5,7,13 & 4 & 6,10 & 8 & 9,25 & 12 & 16 & 18\\\hline
k_{1} & 1 & 4 & 1 & 3 & 1 & 5 & 2 & 1\\\hline
k_{2} & 2 & 2 & 4 & 6 & 3 & 4 & 8 & 6
\end{array}
\]
In particular, if we choose $t,d\in K$ such that $\mathcal{C}_{n,1}\!\left(
t,d\right)  $ is an elliptic curve, then $\mathcal{C}_{n,1}\!\left(
t,d\right)  $ is defined over $K$ and its isogeny class contains $\{
[  \mathcal{C}_{n,i}\!\left(  t,d\right)  ]  _{K}\}_i  $. As a
consequence we obtain the following result:

\begin{mainthms}\label{CorX0}\textit{
Let $n>1$ be an integer such that $X_{0}\!\left(  n\right)  $ has genus $0$
and let $K$ be a field of characteristic $0$ or relatively prime to $6n$.
Suppose further that $E_{1}$ and $E_{2}$ are $n$-isogenous elliptic curves
over $K$. If the $j$-invariants of $E_{1}$ and $E_{2}$ are not both
identically $0$ or $1728$, then there exists a $t\in K$ and $d\in K^{\times
}/(  K^{\times})  ^{2}$ such that $E_{i}$ is $K$-isomorphic to
$\mathcal{C}_{n,k_{i}}\!\left(  t,d\right)  $ (see Table~\ref{ta:curves}). Moreover, the isogeny class of $E_{i}$
contains $\{  [  \mathcal{C}_{n,i}(  t,d) ]
_{K}\}_i  $.}
\end{mainthms}

The construction of these families was motivated by work on explicit isogenies
of prime degree by Cremona, Elkies, Watkins, and Tsukazaki \cite{CremWat,MR1486831,MR3389382}. One motivation of these works was the following question: given an elliptic curve $E_1$ over a field $K$ such that $E_1$ admits a $K$-rational isogeny of prime degree $\ell$, how does one compute the Weierstrass model of the $\ell$-isogenous elliptic curve over $K$? One answer is given through Elkies's Algorithm \cite[Chapter 25, Algorithm 28]{MR2931758}, which outputs the $\ell$-isogenous elliptic curves to $E_1/K$ so long as $\ell>2$ and $K$ is a field having characteristic $0$ or greater than $\ell+2$.

In their preprint \cite{CremWat}, Cremona and Watkins replaced the analytic methods of Elkies with algebraic methods in the case when $X_{0}\!\left(
\ell\right)  $ has genus $0$. Namely, for a field $K$ of characteristic not equal to $2,3,\ell$, they used the Fricke parameterizations to construct elliptic curves $\mathcal{E}_\ell(t)$ over $K(t)$ (notation as in~\cite{MR3389382}) such that the $j$-invariant of $\mathcal{E}_\ell(t)$ is $j_{\ell,1}(t)$. They then computed the kernel polynomial of a $K$-rational $\ell$-isogeny admitted by $\mathcal{E}_{\ell}\!\left(  t\right)  $. 

In his thesis \cite{MR3389382}, Tsukazaki continued the study of $\mathcal{E}_\ell(t)$ and presented an
algorithm for computing the $\ell$-isogenous elliptic curves for a given
elliptic curve $E_1$ over some field $K$ with characteristic not equal to $2,3,\ell$. Specifically, the algorithm solves for
solutions in $K$ to the equation $j_{\ell,1}\!\left(  t\right)  =j\!\left(
E_1\right)  $. For a solution $t_{0}$, the algorithm outputs a $U\in K$ such
that $E_1$ is $K$-isomorphic $\mathcal{E}_{\ell}^U\!\left(  t_{0}\right)  $, where $\mathcal{E}_{\ell}^U\!\left(  t_{0}\right)  $ is the quadratic twist of $\mathcal{E}_{\ell}\!\left(  t_{0}\right)  $ by $U$. We note that special care must be taken when $j(E_1)=0,1728$, since $\mathcal{E}_\ell(t_0)$ is a singular curve. By means of the kernel polynomial of the $\ell$-isogeny of $\mathcal{E}_{\ell}\!\left(  t_{0}\right)  $ and Kohel's formula \cite{MR2695524}, one then obtains an isogenous elliptic
curve $E_2$ such that $j\!\left(  E_2\right)  =j_{\ell,2}\!\left(  t_{0}\right)  $. A
consequence of our work is that Kohel's formulas are no longer required, and in addition, we obtain a subset of the $K$-isogeny
class of $E_1$ upon computing $t_{0},U$. For the precise statements, see
Section~\ref{sec:algorithms}.

We note that Theorem~\ref{CorX0} is also of theoretical interest, as it asserts that for an integer $n>1$ with $X_0(n)$ having genus $0$, the study of $n$-isogenous elliptic curves whose $j$-invariants are not both identically $0$ or $1728$ over a field $K$ of characteristic $0$ or relatively prime to $6n$ is equivalent to understanding the parameterized families $\{\mathcal{C}_{n,i}(t,d)\}_i$. 
In Section~\ref{addred}, we showcase this reasoning to prove the following result:
\begin{mainthms}\label{Thm3}\textit{Let $K$ be a number field. If $E/K$ is an elliptic curve with isogeny class degree divisible by $n\in\left\{  4,6,9\right\} $, then there
exists a quadratic twist $E^{d}$ of $E$ with the property that $E^{d}$ is semistable at all primes $\mathfrak{p}$ of
$K$ such that $\mathfrak{p}\nmid n$.}
\end{mainthms}
\noindent Theorem \ref{Thm3} is analogous to results about elliptic curves with a torsion point. Frey \cite{MR0457444} showed that if $E/K$ has a torsion point of prime order $\ell \geq 5$, then $E$ is semistable at all primes $\mathfrak{p}$ of $K$ such that $\mathfrak{p} \nmid \ell$. In \cite{Barrios2020}, the author showed that if $E$ has a torsion point of order $n\in \{6,8,9 \}$, then $E$ is semistable at all primes $\mathfrak{p}$ of $K$ such that $\mathfrak{p} \nmid n$. We note that the $n=6$ case was originally due to Flexor and Oesterl\'{e} \cite{MR1065153}.

Now suppose $E/\mathbb{Q}$ is an elliptic curve that admits a $\mathbb{Q}$-rational $n$-isogeny. Then there is a
non-cuspidal $\mathbb{Q}$-rational point on $X_{0}\!\left(  n\right)  $. The non-cuspidal $\mathbb{Q}$-rational point on $X_{0}\!\left(  n\right)  $ have been completely described in the literature. In fact, $X_{0}\!\left(  n\right)$ has a non-cuspidal $\mathbb{Q}$-rational point if and only if $n=1,2,\ldots, 19$, $21$, $25$, $27$, $37$, $43$, $67$, $163$. Most notable is the work of Mazur \cite{MR482230}, which gave the equivalence in the case when $n$ is prime. For composite $n$, this is the work of many mathematicians, including Fricke \cite{MR3221641}, Fricke and Klein \cite{MR3838339}, Kenku \cite{MR510395,MR588271,MR549292,MR616546,MR675184}, Klein \cite{MR1509988}, Kubert \cite{MR0434947}, Ligozat \cite{MR0422158}, Mazur \cite{MR482230}, Mazur and V\'{e}lu \cite{MR320010}, and Ogg \cite{MR0337974}. 

We note that if $X_{0}\!\left(  n\right)$ has positive genus, then $X_{0}\!\left(  n\right)$ has a non-cuspidal $\mathbb{Q}$-rational point if and only if $n=11$, $14$, $15$, $17$, $19$, $21$, $27$, $37$, $43$, $67$, $163$. 
Up to twist, these non-cuspidal $\mathbb{Q}$-rational points determine $12$ isogeny classes, consisting of $30$ distinct $\mathbb{Q}$-isomorphism classes of elliptic curves. 
Next, we introduce $30$ new elliptic curves $\mathcal{C}_{n,i}\!\left(  d\right)  $ (see Table~\ref{ta:curves1}) obtained by taking the quadratic twist by $d$ of the minimal twist of the aforementioned $30$ isomorphism classes. In particular, the elliptic curves $\mathcal{C}_{n,i}\!\left(  d\right)  $ parameterize all rational elliptic curves which have isogeny class degree $n$ with $X_0(n)$ having positive genus. Lemma~\ref{LemmaisographsCM} summarizes this discussion.

Mazur conjectured that the size of the isogeny class of $E$ is at most $8$. This was proven by Kenku \cite{MR675184}. The possible isogeny graphs of $E$ were then classified, and are given in \cite[Theorem~1.2]{MR4203041}. 
From Theorem~\ref{CorX0} and Lemma~\ref{LemmaisographsCM}, we obtain an explicit classification of the isogeny graph of~$E$. 
We note that the families $\mathcal{C}_{n,i}(t,d)$ fail to parameterize elliptic curves $E_1$ and $E_2$ whose $j$-invariants are both identically $0$ or $1728$. Parameterized families for these cases are well known, and we introduce an additional four families of elliptic curves $\mathcal{C}_{n,i}^{0}(d)$ and $\mathcal{C}_{n,i}^{1728}(d)$ (see Lemma~\ref{j0or1728}) to obtain our explicit classification of isogeny graphs for rational elliptic curves with isogeny class degree $n>1$:
\begin{mainthms}\label{Thm2}\textit{
Let $E/\mathbb{Q}$ be an elliptic curve with isogeny class degree $n>1$ and set
\[
m=\left\{
\begin{array}
[c]{cl}%
6 & \text{if }n=2\text{ and }j\!\left(  E\right)  =1728,\\
4 & \text{if }n=3\text{ and }j\!\left(  E\right)  =0,\\
2 & \text{otherwise.}%
\end{array}
\right.
\]
Then there exists a $t\in\mathbb{Q}$ and $d\in\mathbb{Q}^{\times}/(\mathbb{Q}
^{\times})^{m}$ such that the isogeny class of $E$ is%
\[
\renewcommand{\arraystretch}{1.3}
\renewcommand{\arraycolsep}{0.7cm}
\begin{array}
[c]{cl}%
\{  [  \mathcal{C}_{2,i}^{1728}(  d)  ]  _{\mathbb{Q}}\}   & \text{if }n=2\text{ and }j(  E)  =1728,\\
\{  [  \mathcal{C}_{3,i}^{0}(  d)  ]  _{\mathbb{Q}
}\}   & \text{if }n=3\text{ and }j(  E)  =0,\\
\{  [  \mathcal{C}_{n,i}(  d)  ]  _{\mathbb{Q}
}\}_{i=1}^{2}  \text{ or }\{  [  \mathcal{C}_{n,i}(  d)
]  _{\mathbb{Q}}\}_{i=3}^{4}   & \text{if }n=11,\\
\{  [  \mathcal{C}_{n,i}(  d)  ]  _{\mathbb{Q}}\}   & \text{if } X_0(n) \text{ has positive genus with } n\neq 11,\\
\{  [  \mathcal{C}_{n,i}(  t,d)  ]  _{\mathbb{Q}
}\}   & \text{otherwise.}
\end{array}
\]
Moreover, the isogeny graph and isogeny matrix of $E$ are as given in Tables \ref{isographs} (if $X_0(n)$ has genus $0$) or \ref{isographsCM} (if $X_0(n)$ has positive genus).}
\end{mainthms}
In particular, Theorem~\ref{Thm2} extends Nitaj's work~\cite{MR1914669}, which explicitly classifies the isogeny class of an elliptic curve $E/\Q$ such that $E(\Q)_{\text{tors}}\cong \Z /9\Z, \Z / 10\Z, \Z / 12\Z,  $ or $\Z / 2\Z \times \Z / 8 \Z$. In Section~\ref{sec:algorithms}, we modify \cite[Algorithms 2]{MR3389382} to make use of Theorems~\ref{CorX0}~and~\ref{Thm2}. 
For instance, applying Algorithm~\ref{Algor2} to the elliptic curve $E:y^2=x^3-15x-22$ (LMFDB \cite{lmfdb} label \href{https://www.lmfdb.org/EllipticCurve/Q/144/a/2}{144.a2}) yields that the isogeny class of $E$ over $\mathbb{Q}$ is $\{[\mathcal{C}_{6,i}(-6,-1)]_\mathbb{Q}\}_{i=1}^4$, and the isogeny graph is obtained from Table~\ref{isographs}. We note that $E$ is $\mathbb{Q}$-isomorphic to $\mathcal{C}_{6,1}(-6,-1)$ and the models of $\{\mathcal{C}_{6,i}(-6,-1)\}_{i=1}^4$ are given by:
\begin{align*}
\mathcal{C}_{6,1}\!\left(  -6,-1\right)    & :y^{2}=x^{3}- 19440x - 1026432,\\
\mathcal{C}_{6,2}\!\left(  -6,-1\right)    & :y^{2}=x^{3}- 2985984,\\
\mathcal{C}_{6,3}\!\left(  -6,-1\right)    & :y^{2}=x^{3}- 174960x + 27713664,\\
\mathcal{C}_{6,4}\!\left(  -6,-1\right)    & :y^{2}=x^{3}+80621568.
\end{align*}
While the families $\mathcal{C}_{n,i}(t,d)$ have not previously appeared in the literature, twists of them were independently found by Gonz\'{a}lez-Jim\'{e}nez for his upcoming work on isogeny-torsion graphs of elliptic curves over $\mathbb{Q}$ \cite{Gonz}.

\section{Preliminaries}

We start by discussing some basic facts about elliptic curves. For
further details, see \cite{MR2024529,MR2514094}. Let $K$ be a field
of characteristic not equal to $2$ or $3$. We say that an elliptic curve $E$ is
defined over $K$ if $E$ is given by a short (affine) Weierstrass model
$E:y^{2}=x^{3}+Ax+B$ for some $A,B\in K$. The set $E\!\left(  K\right)  $ of
$K$-rational points is an abelian group with identity~$\mathcal{O}_{E}$. The
\textit{discriminant} of $E$ is $\Delta=-16\left(  4A^{3}+27B^{2}\right)
\neq0$ and the $j$-invariant of $E$ is $j\!\left(  E\right)  =\frac
{\left(  -48A\right)  ^{3}}{\Delta}$. Now suppose that $E_2%
:y^{2}=x^{3}+A_2x+B_2$ is an elliptic curve over $K$. We say
that $\pi:E\rightarrow E_2$ is an {isogeny} if $\pi$ is a
surjective group homomorphism of elliptic curves. We say that $\pi$ is
\textit{cyclic} if $\ker\pi\cong\mathbb{Z}/n\mathbb{Z}$ for some positive integer $n$. If this is the case, we say that $\pi$ is an
$\mathit{n}$\textit{-isogeny}. We say that $\pi$ is a $\mathit{K}%
$\textit{-rational isogeny} if $\ker\pi$ is $\operatorname*{Gal}\!\left(
\overline{K}/K\right)  $-invariant. When this is the case, we also say that $\pi$ is an \textit{isogeny over $K$} or that $E$ and $E_2$ are \textit{isogenous over $K$}.
In fact, a $K$-rational $n$-isogeny
$\pi:E\rightarrow E_2$ can be written as
\[
\pi\!\left(  x,y\right)  =\left(  f\!\left(  x\right)  ,cy\frac{d}%
{dx}f\!\left(  x\right)  \right)
\]
for some $f\!\left(  x\right)  \in K\!\left(  x\right)  $ and non-zero $c\in
K$. When $c=1$, we say that $\pi$ is a \textit{normalized isogeny} and we have the
following result:

\begin{proposition}
[{\cite[Proposition 4.1]{MR2398793}}]\label{Propnoriso}Let $n>1$ be an integer and let $K$ be a
field of characteristic $0$ or relatively prime to $6n$. Suppose further that
$E_{i}:y^{2}=x^{3}+A_{i}x+B_{i}$ for $i=1,2$ are elliptic curves defined over $K$. If $\pi:E_{1}\rightarrow E_{2}$ is a $K$-rational normalized
$n$-isogeny, then%
\[
\pi\!\left(  x,y\right)  =\left(  \frac{N\!\left(  x\right)  }{D\!\left(
x\right)  },y\frac{d}{dx}\frac{N\!\left(  x\right)  }{D\!\left(  x\right)
}\right)
\]
where
\[
D\!\left(  x\right)  =\prod_{P\in\ker\pi\backslash \{ \mathcal{O}_{E} \} }\left(
x-x_{P}\right)  =x^{n-1}-\sum_{j=2}^{n}\left(  -1\right)  ^{j}\sigma_{j}x^{n-j}\in K\!\left[  x\right]
\]
for some $\sigma_{j}\in K$ with $\sigma_{2}$ denoting the sum of the abscissas
of the points in $\ker\pi\backslash\{\mathcal{O}_{E}\}$. Moreover, $N\!\left(
x\right)  $ is related to $D\!\left(  x\right)  $ through the formula%
\begin{equation}\label{eqnnoriso}
\frac{N\!\left(  x\right)  }{D\!\left(  x\right)  }=nx-\sigma_{2}-\left(
3x^{2}+A_{1}\right)  \frac{\frac{d}{dx}D\!\left(  x\right)  }{D\!\left(
x\right)  }-2\left(  x^{3}+A_{1}x+B_{1}\right)  \frac{d}{dx}\left(
\frac{\frac{d}{dx}D\!\left(  x\right)  }{D\!\left(  x\right)  }\right)  .
\end{equation}

\end{proposition}

Next, let $\left[  E\right]  _{K}$ denote the $K$-isomorphism class of $E$.
The \textit{isogeny class} of $E/K$ is the set%
\[
\mathcal{I}_{E/K}=\left\{  \left[  F\right]  _{K}\mid F\text{ is isogenous to
}E\text{ over }K\right\}  .
\]
The \textit{isogeny graph} of $E$ is the graph whose vertices are elements of
$\mathcal{I}_{E/K}$, and the edges of the graph correspond to isogenies of
prime degree between representatives of the vertices. Now suppose
$\mathcal{I}_{E/K}$ has $m$ elements, and write $\mathcal{I}_{E/K}=\left\{
\left[  E_{i}\right]  _{K}\right\}  _{i=1}^{m}$. The \textit{isogeny matrix
}of $E$ is the symmetric $m\times m$ matrix $M_{E}=\left(  a_{i,k}\right)  $
where $a_{i,k}$ denotes the smallest degree of a cyclic isogeny between
$E_{i}$ and $E_{k}$. The \textit{isogeny class degree} of $E$ is the least common multiple of all the entries in the isogeny matrix of $E$. When $K=\mathbb{Q}$, the largest entry in $M_E$ is equal to the isogeny class degree of $E$, and we have the following result:

\begin{lemma}
Let $E/\mathbb{Q}$ be an elliptic curve with isogeny class degree $n$. Then the
isogeny graph of $E$ is uniquely determined by $n$.
\end{lemma}
\begin{proof}
This follows automatically from \cite[Theorem 1.2]{MR4203041}.
\end{proof}

Let $E_i:y^{2}=x^{3}+A_ix+B_i$ for $i=1,2$ be elliptic curves defined
over $K$. If $j\!\left(  E_1\right)  =j\!\left(  E_2\right)  $, then $E_1$
and $E_2$ are $\overline{K}$-isomorphic and thus $E_1$ is a twist of
$E_2$. In particular, there exists a $d\in K^{\times}/\left(
K^{\times}\right)  ^{m}$ where%
\[
m=\left\{
\begin{array}
[c]{cl}%
2 & j\!\left(  E\right)  \neq0,1728,\\
4 & j\!\left(  E\right)  =1728,\\
6 & j\!\left(  E\right)  =0,
\end{array}
\right.
\]
such that $E_2$ is $K$-isomorphic to%
\[%
\begin{array}
[c]{ll}%
E^{d}_1:y^{2}=x^{3}+d^{2}A_1x+d^{3}B_1 & \text{if }j\!\left(  E_1\right)
\neq0,1728,\\
E^{1728,d}_1:y^{2}=x^{3}+dA_1x & \text{if }j\!\left(  E_1\right)  =1728,\\
E^{0,d}_1:y^{2}=x^{3}+dB_1 & \text{if }j\!\left(  E_1\right)  =0.
\end{array}
\]
We call $E_1^{d}$ the \textit{quadratic twist of $E_1$ by $d$}. If $j\!\left(
E_1\right)  =0$ (resp. $j\!\left(  E_1\right)  =1728$), then the quadratic twist
of $E_1$ by $d$ is the elliptic curve $E_1^{d}=E_1^{0,d^{3}}$ (resp. $E_1^{d}=E_1^{1728,d^{2}}$).

Now let $\pi:E_1\rightarrow E_{2}$ be a $K$-rational $n$-isogeny with
$n>1$. Then we have a $K$-rational $n$-isogeny $\pi^{d}:E_1^{d}\rightarrow
E_{2}^{d}$ and hence, the isogeny class degree of $E_1$ is preserved under
quadratic twist. When $j\!\left(  E_1\right)  =j\!\left(  E_{2}\right)
\in\left\{  0,1728\right\}  $ and $n$ is prime, we have that the following are $K$-rational
$n$-isogenies:
\[%
\begin{array}
[c]{ll}%
\pi^{1728,d}:E_1^{1728,d}\rightarrow E_{2}^{1728,d} & \text{if }j\!\left(
E_1\right)  =j\!\left(  E_{2}\right)  =1728,\\
\pi^{0,d}:E_1^{0,d}\rightarrow E_{2}^{0,d} & \text{if }j\!\left(  E_1\right)
=j\!\left(  E_{2}\right)  =0.
\end{array}
\]

For a positive integer $n$, the $K$-rational points of the modular curve
$X_{0}\!\left(  n\right)  $ (with cusps removed) parameterize $\overline{K}%
$-isomorphism classes of pairs $\left(  E,C\right)  $ where $E$ is an elliptic
curve defined over $K$ and $C$ is a $\operatorname*{Gal}\!\left(  \overline
{K}/K\right)  $-invariant cyclic subgroup of $E$ whose order is $n$. In
particular, if $\pi:E_{1}\rightarrow E_{2}$ is a $K$-rational $n$-isogeny,
then the $\overline{K}$-isomorphism class of the pair $\left(  E_{1},\ker
\pi\right)  $ is a non-cuspidal point of $X_{0}\!\left(  n\right)  \!\left(
K\right)  $. The function field of $X_{0}\!\left(  n\right)  $ is generated by
pairs of functions $\left(  z_1,z_2\right)  $ that satisfy $\Phi
_{n}\!\left(  z_1,z_2\right)  =0$, where $\Phi_{n}\!\left(  X,Y\right)  $
is the classical modular polynomial. Next, let $j\!\left(  \tau\right)  $
denote the modular $j$-function. Then $\Phi_{n}\!\left(  X,Y\right)  $ is the
minimal polynomial of $j\!\left(  n\tau\right)  $ over the field $\mathbb{C}\!\left(  j\!\left(  \tau\right)  \right)  $. From classical results,
$\Phi_{n}\!\left(  X,Y\right)  $ is a symmetric polynomial with integer
coefficients.  The interested reader can find expressions for the modular polynomials
$\Phi_{n}\!\left(  X,Y\right)  $ for $n\leq300$ and primes $n< 1000$ in \cite{Drew}. These
polynomials were computed from the algorithms described in \cite{MR2869057, MR3435725}. Now suppose that $K$ is a field of characteristic $0$ or
relatively prime to $n$. If $\left(  z_1,z_2\right)  \in K^{2}$ is a root
of $\Phi_{n}\!\left(  X,Y\right)  $, then there are elliptic curves $E_{1}$
and $E_{2}$ over $\overline{K}$ such that $j\!\left(  E_{i}\right)  =z_i$, and $E_{1}$ and $E_{2}$ are
$n$-isogenous.

Now suppose that $X_{0}\!\left(  n\right)  $ has genus $0$, i.e.,
$n=1,2,\ldots,10,12,13,16,18,25$. Then the function field of $X_{0}\!\left(
n\right)  $ is generated by the Hauptmodul $h_{n}$ of $X_{0}\!\left(
n\right)  $ (see \cite[Table 2]{MR3084348}). For $n>1$, this gives rise to the
Fricke parameterizations $j_{n,i}\!\left(  t\right)  $ for
$i=1,2$ (see Table \ref{ta:jinv}) which satisfy the following property: if
$\pi:E_{1}\rightarrow E_{2}$ is a $K$-rational $n$-isogeny, then there is a
$t\in K$ such that $j_{n,i}\!\left(  t\right)  =j\!\left(  E_{i}\right)  $
\cite{MR3221641,MR1509993,MR1509988,MR3838339}. In
fact, $\Phi_{n}\!\left(  j_{n,1}\!\left(  t\right)  ,j_{n,2}\!\left(
t\right)  \right)  =0$ for each $t\in K$. We note that Table \ref{ta:jinv} is
reproduced from \cite[Table 3]{MR3084348} and \cite[Tables 4 and 5]%
{MR2514149}, which in turn used results from \cite{MR0376533,MR1486831,MR3221641,MR2341166,MR675184,MR3838339}. The following result summarizes this discussion:
\begin{theorem}\label{thmKF}
Let $n>1$ be an integer and suppose that $X_{0}\!\left(  n\right)  $ has genus
$0$. Let $j_{n,1}\!\left(  t\right)  $ and $j_{n,2}\!\left(  t\right)  $ be as
defined in Table~\ref{ta:jinv}. Suppose further that $E_1$ and $E_2$ are elliptic curves defined over some field $K$ of characteristic $0$ or relatively prime to $n$. If $\pi:E_1\rightarrow E_2$ is a $K$-rational $n$-isogeny,
then there is a $t\in K$ such that the $j$-invariant of $E_i$ is~$j_{n,i}\!\left(  t\right)  $.
\end{theorem}
{\renewcommand*{\arraystretch}{1.55} \begin{longtable}{cC{3in}C{2in}}
\caption{The Fricke Parameterizations}\label{ta:jinv}\\
\hline
$n$ & $j_{n,1}(  t)  $ & $j_{n,2}(  t)  $\\
\hline
\endfirsthead
\hline
$n$ & $j_{n,1}(  t)  $ & $j_{n,2}(  t)  $ \\
\hline
\endhead
\hline
\multicolumn{3}{r}{\emph{continued on next page}}
\endfoot
\hline
\endlastfoot

$2$ & $\frac{(t+256)^{3}}{t^{2}}$ & $\frac{(t+16)^{3}}{t}$\\\hline
$3$ & $\frac{(t+27)(t+243)^{3}}{t^{3}}$ & $\frac{(t+27)(t+3)^{3}}{t}$\\\hline
$4$ & $\frac{(t^{2}+256t+4096)^{3}}{t^{4}(t+16)}$ & $\frac{(t^{2}+16t+16)^{3}%
}{t(t+16)}$\\\hline
$5$ & $\frac{(t^{2}+250t+3125)^{3}}{t^{5}}$ & $\frac{(t^{2}+10t+5)^{3}}{t}$\\\hline
$6$ & $\frac{(t+12)^{3}(t^{3}+252t^{2}+3888t+15552)^{3}}{t^{6}(t+8)^{2}%
(t+9)^{3}}$ & $\frac{(t+6)^{3}(t^{3}+18t^{2}+84t+24)^{3}}{t(t+8)^{3}(t+9)^{2}%
}$\\\hline
$7$ & $\frac{(t^{2}+13t+49)(t^{2}+245t+2401)^{3}}{t^{7}}$ & $\frac
{(t^{2}+13t+49)(t^{2}+5t+1)^{3}}{t}$\\\hline
$8$ & $\frac{(t^{4}+240t^{3}+2144t^{2}+3840t+256)^{3}}{t(t-4)^{8}(t+4)^{2}}$ &
$\frac{(t^{4}-16t^{2}+16)^{3}}{t^{2}(t^{2}-16)}$\\\hline
$9$ & $\frac{(t+6)^{3}(t^{3}+234t^{2}+756t+2160)^{3}}{(t-3)^{8}(t^{3}-27)}$ &
$\frac{t^{3}(t^{3}-24)^{3}}{t^{3}-27}$\\\hline
$10$ & $\frac{(t^{6}+236t^{5}+1440t^{4}+1920t^{3}+3840t^{2}+256t+256)^{3}%
}{t^{2}(t-4)^{10}(t+1)^{5}}$ & $\frac{(t^{6}-4t^{5}+16t+16)^{3}}%
{t^{5}(t+1)^{2}(t-4)}$\\\hline
$12$ & $\frac{(t^{2}+6t-3)^{3}(t^{6}+234t^{5}+747t^{4}+540t^{3}-729t^{2}%
-486t-243)^{3}}{t^{3}(t-3)^{12}(t-1)(t+1)^{4}(t+3)^{3}}$ & $\frac
{(t^{2}-3)^{3}(t^{6}-9t^{4}+3t^{2}-3)^{3}}{t^{4}(t^{2}-9)(t^{2}-1)^{3}}$\\\hline
$13$ & $\frac{(t^{2}+5t+13)(t^{4}+247t^{3}+3380t^{2}+15379t+28561)^{3}}%
{t^{13}}$ & $\frac{(t^{2}+5t+13)(t^{4}+7t^{3}+20t^{2}+19t+1)^{3}}{t}$\\\hline
$16$ & $\frac{1}{t(t-2)^{16}(t+2)^{4}(t^{2}+4)}(t^{8}+240t^{7}+2160t^{6}+6720t^{5}+17504t^{4}+26880t^{3}%
+34560t^{2}+15360t+256)^{3}$ & $\frac
{(t^{8}-16t^{4}+16)^{3}}{t^{4}(t^{4}-16)}$\\\hline
$18$ & $\frac{(t^{3}+6t^{2}+4)^{3}}{t^{2}(t-2)^{18}(t+1)^{9}%
(t^{2}-t+1)(t^{2}+2t+4)^{2}}(t^{9}+234t^{8}+756t^{7}+2172t^{6}
+1872t^{5}+3024t^{4}+48t^{3}+3744t^{2}+64)^{3}$ & $\frac{(t^{3}-2)^{3}(t^{9}-6t^{6}%
-12t^{3}-8)^{3}}{t^{9}(t^{3}-8)(t^{3}+1)^{2}}$\\\hline
$25$ & $\frac{1}{(t-1)^{25}
(t^{4}+t^{3}+6t^{2}+6t+11)}(t^{10}+240t^{9}+2170t^{8}+8880t^{7}+34835t^{6}+83748t^{5}%
+206210t^{4}+313380t^{3}+503545t^{2}+424740t+375376)^{3}$ & $\frac{1}{t^{5}+5t^{3}+5t-11}(t^{10}+10t^{8}+35t^{6}-12t^{5}%
+50t^{4}-60t^{3}+25t^{2}-60t+16)^{3}$
\end{longtable}}

\section{Isogenous Families of Elliptic Curves}
For an integer $n>1$ such that $X_0(n)$ has genus~$0$, let $\mathcal{C}_{n,i}\!\left(  t,d\right)  $ be as defined in Table~\ref{ta:curves} and let $K$ be a field of characteristic $0$ or relatively prime to $6n$. Our first result shows that for a fixed~$n$, the elliptic curves $\mathcal{C}_{n,i}\!\left(  t,d\right)  $ are isogenous over $K(t,d)$. In the process, we show that the isogenies are normalized, and their explicit expressions are found in \cite{GitHubExplicitIsogenies}. The proofs in this section rely on computer verification, which was done on SageMath \cite{sagemath}. After establishing that the elliptic curves $\mathcal{C}_{n,i}\!\left(  t,d\right)  $ are isogenous over $K(t,d)$, we investigate those values of $t\in \overline{K}$ for which  $\mathcal{C}_{n,i}\!\left(  t,1\right)  $ is a singular curve. We conclude the section by proving Theorem~\ref{CorX0}.

\begin{proposition}\label{thmcalC}
Suppose $n>1$ is an integer such that $X_{0}\!\left(  n\right)  $ has genus
$0$ and let $K$ be a field of characteristic $0$ or relatively prime to $6n$. Let $j_{n,i}\!\left(t\right)  $ and $\mathcal{C}_{n,i}\!\left(  t,d\right)  $ be as defined in Tables \ref{ta:jinv} and \ref{ta:curves},
respectively. For each $n$, define $k_{1}$ and $k_{2}$ as:%
\[%
\begin{array}
[c]{ccccccccc}%
n & 2,3,5,7,13 & 4 & 6,10 & 8 & 9,25 & 12 & 16 & 18\\\hline
k_{1} & 1 & 4 & 1 & 3 & 1 & 5 & 2 & 1\\\hline
k_{2} & 2 & 2 & 4 & 6 & 3 & 4 & 8 & 6
\end{array}
\]
Then the elliptic curves $\mathcal{C}_{n,i}\!\left(  t,d\right)  $ are defined
over $K
\!\left(  t,d\right)  $ and satisfy the following properties:
\begin{itemize}

\item[$\left(  i\right)  $] The $j$-invariant of $\mathcal{C}_{n,k_{i}}\!\left(
t,d\right)  $ is $j_{n,i}\!\left(  t\right)  $ for $i=1,2$;

\item[$\left(  ii\right)  $] The elliptic curves $\mathcal{C}_{n,i}\!\left(
t,d\right)  $ are pairwise non-isomorphic over $K
\!\left(  t,d\right)  $;

\item[$\left(  iii\right)  $] The isogeny class of $\mathcal{C}_{n,1}\!\left(
t,d\right)  $ over $K\!\left(  t,d\right)  $ contains the set $\{  [  \mathcal{C}%
_{n,i}\!\left(  t,d\right)  ]_{K(t,d)}  \}  _{i}$;

\item[$\left(  iv\right)  $] The partial isogeny graph corresponding to $\{
[  \mathcal{C}_{n,i}\!\left(  t,d\right) ]_{K(t,d)} \}  _{i}$ is
given in Table \ref{isographs}, where $\psi_{n,i}^{(m)}$ corresponds to a degree $m$ cyclic isogeny $\psi_{n,i}:\mathcal{C}_{n,k}\!\left(  t,d\right)  \rightarrow
\mathcal{C}_{n,l}\!\left(  t,d\right)  $ defined over $K
\!\left(  t,d\right)  $.
\end{itemize}
\end{proposition}
{\begingroup
\renewcommand{\arraystretch}{1.3}
 \begin{longtable}{ccc}
 	\caption{Isogeny graph and matrix of $\left\{
\left[  \mathcal{C}_{n,i}\!\left(  t,d\right)  \right]_{K(t,d)}  \right\}  _{i}$}\label{isographs}\\
	\hline
	$n$  & Isogeny Graph & Isogeny Matrix \\
	\hline

	\endfirsthead
	\hline
	$n$  & Isogeny Graph & Isogeny Matrix\\
	\hline
	\endhead
	\hline

	\multicolumn{3}{r}{\emph{continued on next page}}
	\endfoot
	\hline
	\endlastfoot

$2,3,5,7,13$& \begin{tikzcd}
\mathcal{C}_{n,1} \arrow[r, "\psi_{n,1}^{(n)}"] & \mathcal{C}_{n,2}
\end{tikzcd} & \resizebox{.08\textwidth}{!}{$\left(
\begin{array}
[c]{cc}%
1 & n\\
n & 1
\end{array}
\right)  $}	\\\hline

$9,25$ & \begin{tikzcd}
\mathcal{C}_{n,1} \arrow[r, "\psi_{n,1}^{(\sqrt{n})}"] & \mathcal{C}_{n,2} \arrow[r, "\psi_{n,2}^{(\sqrt{n})}"] & \mathcal{C}_{n,3}
\end{tikzcd} &\resizebox{.12\textwidth}{!}{$\left(
\begin{array}
[c]{ccc}%
1 & \sqrt{n} & n\\
\sqrt{n} & 1 & \sqrt{n}\\
n & \sqrt{n} & 1
\end{array}
\right)  $} \\\hline

$6,10$ & \begin{tikzcd}
\mathcal{C}_{n,1} \arrow[r, "\psi_{n,1}^{(2)}"] \arrow[d, "\psi_{n,2}^{(\frac{n}{2})}"] & \mathcal{C}_{n,2} \arrow[d, "\psi_{n,4}^{(\frac{n}{2})}"] \\
\mathcal{C}_{n,3} \arrow[r, "\psi_{n,3}^{(2)}"']               & \mathcal{C}_{n,4}               
\end{tikzcd} & \resizebox{.15\textwidth}{!}{$\left(
\begin{array}
[c]{cccc}%
1 & 2 & \frac{n}{2} & n\\
2 & 1 & n & \frac{n}{2}\\
\frac{n}{2} & n & 1 & 2\\
n & \frac{n}{2} & 2 & 1
\end{array}
\right)  $} \\\hline

$4$ & \begin{tikzcd}
  & \mathcal{C}_{4,2}                                  &                    \\
  & \mathcal{C}_{4,1} \arrow[u, "\psi_{4,1}^{(2)}"'] \arrow[ld, "\psi_{4,2}^{(2)}"'] &                    \\
\mathcal{C}_{4,3} &                                    & \mathcal{C}_{4,4} \arrow[lu, "\psi_{4,3}^{(2)}"']
\end{tikzcd} & \resizebox{.15\textwidth}{!}{$\left(
\begin{array}
[c]{cccc}%
1 & 2 & 2 & 2\\
2 & 1 & 4 & 4\\
2 & 4 & 1 & 4\\
2 & 4 & 4 & 1
\end{array}
\right)  $} \\\hline

$8$ & \begin{tikzcd}
\mathcal{C}_{8,2}                  &                                   &                                    & \mathcal{C}_{8,5}   \\
                   & \mathcal{C}_{8,1}   \arrow[lu, "\psi_{8,1}^{(2)}"'] \arrow[r, "\psi_{8,3}^{(2)}"] & \mathcal{C}_{8,4}   \arrow[ru, "\psi_{8,4}^{(2)}"] \arrow[rd, "\psi_{8,5}^{(2)}"] &   \\
\mathcal{C}_{8,3}   \arrow[ru, "\psi_{8,2}^{(2)}"] &                                   &                                    & \mathcal{C}_{8,6}  
\end{tikzcd} & \resizebox{.2\textwidth}{!}{$\left(
\begin{array}
[c]{cccccc}%
1 & 2 & 2 & 2 & 4 & 4\\
2 & 1 & 4 & 4 & 8 & 8\\
2 & 4 & 1 & 4 & 8 & 8\\
2 & 4 & 4 & 1 & 2 & 2\\
4 & 8 & 8 & 2 & 1 & 4\\
4 & 8 & 8 & 2 & 4 & 1
\end{array}
\right)  $}  \\\hline

$12$ & \begin{tikzcd}
                                  & \mathcal{C}_{12,3} \arrow[r, "\psi_{12,9}^{(3)}"]                                 & \mathcal{C}_{12,4}                 &                    \\
                                  & \mathcal{C}_{12,1} \arrow[r, "\psi_{12,1}^{(3)}"] \arrow[rd, "\psi_{12,4}^{(2)} \text{ } \text{ } \text{ } \text{ }"'] \arrow[u, "\psi_{12,10}^{(2)}"] & \mathcal{C}_{12,2} \arrow[u, "\psi_{12,8}^{(2)}"'] &                    \\
\mathcal{C}_{12,5} \arrow[ru, "\psi_{12,5}^{(2)}"] \arrow[r, "\psi_{12,6}^{(3)}"'] & \mathcal{C}_{12,6} \arrow[ru, "\psi_{12,7}^{(2)} \text{ } \text{ } \text{ }"]                               & \mathcal{C}_{12,7} \arrow[r, "\psi_{12,3}^{(3)}"']  & \mathcal{C}_{12,8} \arrow[lu, "\psi_{12,2}^{(2)}"']
\end{tikzcd} & \resizebox{.24\textwidth}{!}{$\left(
 \begin{array}
[c]{cccccccc}%
1 & 3 & 2 & 6 & 2 & 6 & 2 & 6\\
3 & 1 & 6 & 2 & 6 & 2 & 6 & 2\\
2 & 6 & 1 & 3 & 6 & 12 & 4 & 12\\
6 & 2 & 3 & 1 & 12 & 4 & 12 & 4\\
2 & 6 & 6 & 12 & 1 & 3 & 4 & 12\\
6 & 2 & 12 & 4 & 3 & 1 & 12 & 4\\
2 & 6 & 4 & 12 & 4 & 12 & 1 & 3\\
6 & 2 & 12 & 4 & 12 & 4 & 3 & 1
\end{array}
\right)  $} \\\hline

$16$ &  \adjustbox{scale=0.95}{\begin{tikzcd}
  & \mathcal{C}_{16,2} \arrow[d, "\psi_{16,2}^{(2)}"]                   &                                   & \mathcal{C}_{16,7}                                &   \\
  & \mathcal{C}_{16,1} \arrow[ld, "\psi_{16,1}^{(2)}"'] \arrow[rd, "\psi_{16,3}^{(2)}"] &                                   & \mathcal{C}_{16,6} \arrow[u, "\psi_{16,6}^{(2)}"'] \arrow[rd, "\psi_{16,7}^{(2)}"] &   \\
\mathcal{C}_{16,3} &                                    & \mathcal{C}_{16,4} \arrow[d, "\psi_{16,4}^{(2)}"] \arrow[ru, "\psi_{16,5}^{(2)}"'] &                                   & \mathcal{C}_{16,8} \\
  &                                    & \mathcal{C}_{16,5}                                 &                                   &  
\end{tikzcd}} & \resizebox{.24\textwidth}{!}{$\left(
\begin{array}
[c]{cccccccc}%
1 & 2 & 2 & 2 & 4 & 4 & 8 & 8\\
2 & 1 & 4 & 4 & 8 & 8 & 16 & 16\\
2 & 4 & 1 & 4 & 8 & 8 & 16 & 16\\
2 & 4 & 4 & 1 & 2 & 2 & 4 & 4\\
4 & 8 & 8 & 2 & 1 & 4 & 8 & 8\\
4 & 8 & 8 & 2 & 4 & 1 & 2 & 2\\
8 & 16 & 16 & 4 & 8 & 2 & 1 & 4\\
8 & 16 & 16 & 4 & 8 & 2 & 4 & 1
\end{array}
\right)  $} \\\hline

$18$ & \begin{tikzcd}
\mathcal{C}_{18,1} \arrow[r, "\psi_{18,1}^{(3)}"] \arrow[d, "\psi_{18,2}^{(2)}"] & \mathcal{C}_{18,3} \arrow[d, "\psi_{18,4}^{(2)}"] \arrow[r, "\psi_{18,5}^{(3)}"] & \mathcal{C}_{18,5} \arrow[d, "\psi_{18,6}^{(2)}"] \\
\mathcal{C}_{18,2} \arrow[r, "\psi_{18,3}^{(3)}"']                & \mathcal{C}_{18,4} \arrow[r, "\psi_{18,7}^{(3)}"']                & \mathcal{C}_{18,6}               
\end{tikzcd} & \resizebox{.18\textwidth}{!}{$\left(
\begin{array}
[c]{cccccc}%
1 & 2 & 3 & 6 & 9 & 18\\
2 & 1 & 6 & 3 & 18 & 9\\
3 & 6 & 1 & 2 & 3 & 6\\
6 & 3 & 2 & 1 & 6 & 3\\
9 & 18 & 3 & 6 & 1 & 2\\
18 & 9 & 6 & 3 & 2 & 1
\end{array}
\right)  $} 

\end{longtable}
\endgroup}

\begin{proof}
Observe that each $\mathcal{C}_{n,l}(t,d)$ is defined over $K(t,d)$. For each $n$, let $L_{n,i}$ be as defined in \cite[models.sage]{GitHubExplicitIsogenies}. In particular, $L_{n,i}\in K(t,d,x)$. Let $\psi_{n,i}:\mathcal{C}_{n,k}\!\left(
t,d\right)  \rightarrow\mathcal{C}_{n,l}\!\left(  t,d\right)  $ be defined by
$\psi_{n,i}\!\left(  x,y\right)  =\left(  L_{n,i},y\frac{d}{dx}L_{n,i}\right)
$ for the tuples $\left(  n,i,k,l\right)  \neq\left(  12,2,8,2\right)  $
appearing in Table \ref{isographs}. First, we show that each $\psi_{n,i}$ is a
normalized isogeny. We sketch the case when $n=4$, and refer the reader to
\cite[Verification\_of\_calC.ipynb]{GitHubExplicitIsogenies} for the
verification of the remaining $n$.
Observe that%
\begin{align*}
L_{4,1}  & =\frac{x^2+1296d^2t^2 + 20736d^2t + 24dtx + 192dx }{x+24dt + 192d },\\
L_{4,2}  & =\frac{x^2-20736d^2t - 12dtx + 192dx}{x-12dt - 192d},\\
L_{4,3}  & =\frac{x^2+81d^2t^2 + 6dtx + 192dx }{x+6dt + 192d}.
\end{align*}
It is then verified that each
$\psi_{4,i}:\mathcal{C}_{4,k}\!\left(  t,d\right)  \rightarrow\mathcal{C}%
_{4,l}\!\left(  t,d\right)  $ is a normalized isogeny since
\begin{align*}
\left(  y\frac{d}{dx}L_{4,i}\right)  ^{2}  & =\left(  x^{3}+d^2\mathcal{A}%
_{4,k}x+d^3\mathcal{B}_{4,k}\right)  \left(  \frac{d}{dx}L_{4,i}\right)  ^{2}\\
& =L_{4,i}^{3}+d^2\mathcal{A}_{4,l}L_{4,i}+d^3\mathcal{B}_{4,l}.
\end{align*}

For $\left(  n,i,k,l\right)  =\left(  12,2,8,2\right)  $, let
$\widehat{\mathcal{C}}_{12,2}\!\left(  t,d\right)  $ be the elliptic curve
obtained from $\mathcal{C}_{12,2}\!\left(  t,d\right)  $ via the admissible
change of variables $(x,y)\longmapsto(2^{-4}x,2^{-8}y)$. In
particular, we have a $K\!\left(  t,d\right)  $-isomorphism $\phi:\widehat{\mathcal{C}}_{12,2}\!\left(
t,d\right)  \rightarrow\mathcal{C}_{12,2}\!\left(  t,d\right)  $. Define
$\widehat{\psi}_{12,2}:\mathcal{C}_{12,8}\!\left(  t,d\right)  \rightarrow
\widehat{\mathcal{C}}_{12,2}\!\left(  t,d\right)  $ by $\widehat{\psi}_{12,2}\!\left(
x,y\right)  =\left(  L_{12,2},y\frac{d}{dx}L_{12,2}\right)  $. Then loc. cit. verifies that
$\widehat{\psi}_{12,2}$ is a normalized isogeny. In particular, $\psi
_{12,2}=\phi\circ\widehat{\psi}_{12,2}$ is an isogeny. 

Note that each
$\psi_{n,i}$ is an isogeny defined over $K\!\left(  t,d\right)  $ since each $L_{n,i}\in K\!\left(  t,d,x \right)  $. In addition, the verification shows that the partial
isogeny graph corresponding to $\{  [  \mathcal{C}_{n,i}\!\left(
t,d\right)  ]_{K(t,d)}  \}  _{i}$ is as given in Table~\ref{isographs}.

Lastly, loc. cit. verifies that $j(\mathcal{C}_{n,k_i}(t,d))=j_{n,i}(t)$ for each $n$ and $i=1,2$. This shows that the $K\!\left(  t,d\right)  $-isogeny class $\mathcal{I}_{\mathcal{C}_{n,1}\!\left(
t,d\right)  }$ of $\mathcal{C}_{n,1}\!\left(  t,d\right)  $ contains $\{
[  \mathcal{C}_{n,i}\!\left(  t,d\right)  ]_{K(t,d)} \}  _{i}$, which
concludes the proof.
\end{proof}

\begin{remark}
With notation as in the above proof, let $\psi_{n,i}:\mathcal{C}%
_{n,k}\!\left(  t,d\right)  \rightarrow\mathcal{C}_{n,l}\!\left(  t,d\right)
$ be the normalized isogeny defined by $\psi_{n,i}\!\left(  x,y\right)
=\left(  L_{n,i},y\frac{d}{dx}L_{n,i}\right)  $. By Proposition~\ref{Propnoriso}, $L_{n,i}=\frac{N_{n,i}}{D_{n,i}} $ for some $N_{n,i},D_{n,i} \in K(t,d)[x]$ such that \eqref{eqnnoriso} holds. In what follows, we explicitly give  $N_{n,i},D_{n,i}$. To this end, let \texttt{sigman\_i} and
\texttt{Dn\_i} be as defined in \cite[Isogenies\_for\_calC.sage]{GitHubExplicitIsogenies}. We denote these expressions by $\sigma_{n,i}$ and
$D_{n,i}$, respectively. Set $N_{n,i}=L_{n,i}D_{n,i}$ and observe that $\sigma_{n,i},D_{n,i},N_{n,i}\in K[t,d,x]$. Then \cite[Verification\_of\_calC.ipynb]{GitHubExplicitIsogenies} verifies that \eqref{eqnnoriso} holds for each $n,i$ since
\[
L_{n,i}=\frac{N_{n,i}}{D_{n,i}}=nx-\sigma_{n,i}-\left(
3x^{2}+\mathcal{A}_{n,k}\right)  \frac{\frac{d}{dx}D_{n,i} }{D_{n,i} }-2\left(  x^{3}+\mathcal{A}_{n,k}x+\mathcal{B}_{n,k}\right)  \frac{d}{dx}\left(
\frac{\frac{d}{dx}D_{n,i} }{D_{n,i}  }\right)  .
\]
In fact, our construction of the families $\mathcal{C}_{n,i}(t,d)$ relied on Proposition~\ref{Propnoriso} and knowledge of the kernel polynomial of $\mathcal{C}_{n,k_1}(t,d)$.
\end{remark}

Now suppose that $t\in\overline{K}$ such that $j_{n,i}(t)$ is defined for $i=1,2$. By Proposition~\ref{thmcalC}, the $j$-invariant of $\mathcal{C}_{n,k_i}(t,1)$ is $j_{n,i}(t)$, provided that $\mathcal{C}_{n,k_i}(t,1)$ is an elliptic curve. The next lemma investigates those $t\in\overline{K}$ for which $\mathcal{C}_{n,k_i}(t,d)$ is a singular curve.

\begin{lemma}\label{thelemma}
Suppose $n>1$ is an integer such that $X_{0}\!\left(  n\right)  $ has genus
$0$ and let $K$ be a field of characteristic $0$ or relatively prime to $6n$.
Let $t\in\overline{K}$ such that $j_{n,l}\!\left(  t\right)  $ is defined for $l=1,2$. If $\mathcal{C}%
_{n,i}\!\left(  t,1\right)  $ is a singular curve for some $i$, then $j_{n,1}\!\left(
t\right)  =j_{n,2}\!\left(  t\right)  \in\left\{  0,1728\right\}  $.
\end{lemma}

\begin{proof}
Let $R_{n}=\left\{  t\in\overline{K}\mid j_{n,l}\!\left(  t\right)  \in\overline{K}\text{ and }\mathcal{C}_{n,1}\!\left(  t,1\right)  \text{ is a singular
curve}\right\}  $. By inspection of the discriminants of $\mathcal{C}_{n,i}(t,1)$, we have that $\mathcal{C}_{n,1}(t,1)$ is a singular curve if and only if $\mathcal{C}_{n,2}(t,1)$ is a singular curve. Moreover,
\[
R_{n}=\left\{
\begin{array}
[c]{ll}%
\varnothing & \text{if }n=4,6,8,9,12,16,18,\\
\left\{  -64\right\}   & \text{if }n=2,\\
\left\{  -27\right\}   & \text{if }n=3,\\
\left\{  t \in \overline{K} \mid t^2 + 22t + 125=0 \right\}   & \text{if }n=5,\\
\left\{  t \in \overline{K} \mid t^2 + 13t + 49=0  \right\}   & \text{if }n=7,\\
\left\{  t \in \overline{K} \mid t^2 + 4 = 0  \right\}   & \text{if }n=10,25,\\
\left\{  t \in \overline{K} \mid (t^2 + 6t + 13)(t^2 + 5t + 13)=0  \right\}   & \text{if
}n=13.
\end{array}
\right.
\]
Next, let $t\in R_{5}$ so that $t^{2}+22t+125=0$. Now observe that%
\[
j_{n,1}\!\left(  t\right)    =\frac{\left(  t^{2}+250t+3125\right)  ^{3}%
}{t^{5}} = 1728\frac{\left(  19t+250 \right)  ^{3}}{t^{5}} = 1728
\]
since $\left(  19t+250\right)  ^{3}-t^{5}=-\left(  t^{3}-22t^{2}%
-6500t-125000\right)  (t^{2}+22t+125)=0$ implies that $\left(  19t+250\right)
^{3}/t^{5}=1$. Similar arguments (see \cite[Verification\_of\_Lemma\_3\_3.ipynb]{GitHubExplicitIsogenies}) show that for each $t\in R_{n}$ where $n=2,3,5,7,10,25$, 
\[
j_{n,1}\!\left(  t\right)  =j_{n,2}\!\left(  t\right)  =\left\{
\begin{array}
[c]{cl}%
0 & \text{if }n=3,7,\\
1728 & \text{if }n=2,5,10,25.
\end{array}
\right.
\]
The lemma now follows since
\begin{align*}
j_{13,1}\!\left(  t\right)  =j_{13,2}\!\left(  t\right)&=0  \quad \qquad\text{if }t^2+6t+13=0,\\
j_{13,1}\!\left(  t\right)  =j_{13,2}\!\left(  t\right)&=1728  \hspace{-0.4em} \qquad \text{if }t^2+5t+13=0.
\qedhere \end{align*} 
\end{proof}

As a consequence of the above results, we obtain Theorem~\ref{CorX0}:
\begin{proof}[Proof of Theorem~\ref{CorX0}.]
Since $E_{1}$ and $E_{2}$ are $n$-isogenous
over $K$, we have by Theorem \ref{thmKF} that the $j$-invariant of $E_{i}$ is
$j_{n,i}\!\left(  t\right)  $ for some $t \in K$. It follows by Proposition~\ref{thmcalC} and
Lemma~\ref{thelemma}, that $E_{i}$ is $\overline{K}$-isomorphic to
$\mathcal{C}_{n,k_{i}}\!\left(  t,1\right)  $. In particular, $E_{i}$ is a
twist of $\mathcal{C}_{n,k_{i}}\!\left(  t,1\right)  $. 
Since $E_1$ and $E_2$ are $n$-isogenous over $K$ such that their $j$-invariants are not both identically $0$ or $1728$,
we have by \cite[Proposition X.5.4]{MR2514094} that there is a $d\in
K^{\times}/\left(  K^{\times}\right)  ^{2}$ such that $E_{i}$ is
$K$-isomorphic to a quadratic twist by $d$ of $\mathcal{C}_{n,k_{i}}\!\left(
t,1\right)  $. In particular, $E_{i}$ is $K$-isomorphic to $\mathcal{C}%
_{n,k_{i}}(t,d)$. By Proposition~\ref{thmcalC}, the isogeny class $\mathcal{I}_{E/K}$ of $E$ contains $\{  [  \mathcal{C}_{n,i}\!\left(  t,d\right)  ]
_{K}\}_i  $. 
\end{proof}

\section{Isogenies and Semistability}\label{addred}
In this section, we prove Theorem~\ref{Thm3}. To this end, let $n\in\left\{  4,6,9\right\}  $ and define $F_{n}\!\left(  a,b\right)
:y^{2}+a_{1}xy+a_{3}y=x^{3}+a_{2}x^{2}+a_{4}x$ for $a_{1},a_{2},a_{3},a_{4}$ as given below:
\begin{equation}
\renewcommand{\arraystretch}{1.2}
\renewcommand{\arraycolsep}{0.4cm}
\begin{array}
[c]{ccccc}%
n & a_{1} & a_{2} & a_{3}& a_{4}\\\hline
4 & 0 & b-16a & 0 & -16ab \\\hline
6 & 36a+5b & 2b(9a+b) & 9b(8a+b)(9a+b)& 0\\\hline
9 & 3(6a+b) & 0 & (b-3a)^{3} & 0
\end{array}
\end{equation}
Next, let $\alpha_{n}$ and $\gamma_{n}$ be defined by the following expressions:
\begin{equation}
\renewcommand{\arraystretch}{1.2}
\renewcommand{\arraycolsep}{0.1cm}
\begin{array}
[c]{ccc}%
n & \alpha_{n} & \gamma_{n}\\\hline
4 &16(256a^2 + 16ab + b^2) & 4096 a^2b^2 (16a + b)^2 \\\hline
6 & 9(12a+b)(15552a^{3}+3888a^{2}b+252ab^{2}+b^{3}) & 729ab^{6}%
(8a+b)^{2}(9a+b)^{3}\\\hline
9 & 9(6a+b)(2160a^{3}+756a^{2}b+234ab^{2}+b^{3}) & 729a(-3a+b)^{9}%
(9a^{2}+3ab+b^{2})%
\end{array}
\label{alpgam}%
\end{equation}

\begin{lemma}
\label{intmodel}Let $K$ be a field of characteristic $0$. Let $n\in\left\{
4,6,9\right\}  $ and set $d_4=a$ and $d_n=1$ for $n=6,9$.
Then the elliptic curves $\mathcal{C}_{n,1}\!\left(  \frac{b}%
{a},d_n\right)  $ and $F_{n}\!\left(  a,b\right)  $ are $K\!\left(  a,b\right)
$-isomorphic. Moreover, the invariants $c_{4}$ and $\Delta$ of $F_{n}\!\left(
a,b\right)  $ are $\alpha_{n}$ and $\gamma_{n}$, respectively.
\end{lemma}

\begin{proof}
This is verified in \cite[section\_4.ipynb]{GitHubExplicitIsogenies}.
\end{proof}

\begin{lemma}\label{GCD}
For $n\in\left\{  4,6,9\right\}  $, let $\alpha_{n}$ and $\gamma_{n}$ be as
given in (\ref{alpgam}). Then there exists $\mu_{n},\nu_{n}\in\mathbb{Z}\!\left[  a,b,r,s\right]  $ such that the following identities hold in $\mathbb{Z}\!\left[  a,b,r,s\right]  $:%
\begin{equation}
\renewcommand{\arraystretch}{1.2}
\renewcommand{\arraycolsep}{0.4cm}
\begin{array}
[c]{cccc}%
n & 4 & 6 & 9\\\hline
\mu_{n}\alpha_{n}+\nu_{n}\gamma_{n}  & 2^{28}(ra^6+sb^6) & 2^{16}3^{24}\left(  ra^{15}%
+sb^{15}\right)   & 3^{39}\left(
ra^{15}+sb^{15}\right)  %
\end{array}
\end{equation}
In particular, if $K$ is a number field or a local field with ring of integers
$R_{K}$ and $a,b\in R_{K}$ are coprime elements, then the ideal $\left(
\alpha_{n}\!\left(  a,b\right)  \right)  +\left(  \gamma_{n}\!\left(
a,b\right)  \right)  $ is contained in the principal ideal~$nR_{K}$.
\end{lemma}

\begin{proof}
Explicit equations for $\mu_{n},\nu_{n}\in%
\mathbb{Z}
\!\left[  a,b,r,s\right]  $ are found in \cite[models.sage]%
{GitHubExplicitIsogenies}. The result is verified in \cite[section\_4.ipynb]%
{GitHubExplicitIsogenies}.
\end{proof}

The next result guarantees that Theorem~\ref{CorX0} can be invoked whenever $E_1$ and $E_2$ are $n$-isogenous elliptic curves over $K$ with $n\in \{4,6,9\}$.

\begin{lemma}
\label{Lemma01728}Let $K$ be a number field and let $E_{1}$ and $E_{2}$ be
$n$-isogenous elliptic curves over~$K$ where $n\in\left\{  4,6,9\right\}  $.
If $j\!\left(  E_{i}\right)  \in\left\{  0,1728\right\}  $ for some $i=1,2$,
then $j\!\left(  E_{1}\right)  \neq j\!\left(  E_{2}\right)  $.
\end{lemma}

\begin{proof}
Without loss of generality we may assume that $j\!\left(  E_{1}\right)
\in\left\{  0,1728\right\}  $. By Theorem \ref{thmKF}, there is a $t\in K$
such that the $j$-invariant of $E_{i}$ is $j_{n,i}\!\left(  t\right)  $. Then
\cite[section\_4.ipynb]{GitHubExplicitIsogenies} verifies that if $j_{n,1}\!\left(
t\right)  \in\left\{  0,1728\right\}  $, then $j_{n,1}\!\left(  t\right)  \neq
j_{n,2}\!\left(  t\right)  $.
\end{proof}

We are now ready to prove Theorem \ref{Thm3}:

\begin{proof}
[Proof of Theorem \ref{Thm3}]Since $E$ admits a $K$-rational $n$-isogeny over
$K$ with $n\in\left\{  4,6,9\right\}  $, there is an elliptic curve
$E^{\prime}$ that is $n$-isogenous to $E$ over $K$. By Lemma \ref{Lemma01728},
the $j$-invariants of $E$ and $E^{\prime}$ are not both identically $0$ or
$1728$. By Theorem \ref{CorX0}, there exists a $t\in K$ and $d\in K^{\times
}/\left(  K^{\times}\right)  ^{2}$ such that $E$ is $K$-isomorphic to
$\mathcal{C}_{n,k_{1}}\!\left(  t,d\right)  $. In particular, $E^{d}$ is
$K$-isomorphic to $\mathcal{C}_{n,k_{1}}\!\left(  t,1\right)  $. Moreover,
$\mathcal{C}_{n,1}\!\left(  t,1\right)  $ is isogenous to $\mathcal{C}%
_{n,k_{1}}\!\left(  t,1\right)  $ over $K$. If $\mathcal{C}_{n,1}\!\left(
t,1\right)  $ is semistable, then we are done. So suppose $\mathcal{C}%
_{n,1}\!\left(  t,1\right)  $ has additive reduction at some prime
$\mathfrak{p}$ of $K$. Let $R_{K_{\mathfrak{p}}}$ denote the ring of integers
of $K_{\mathfrak{p}}$. Now suppose that $t=\frac{b}{a}$ for some coprime elements $a$
and $b$ of $R_{K_{\mathfrak{p}}}$. By Lemma \ref{intmodel}, $\mathcal{C}%
_{n,1}\!\left(  t,1\right)  $ is $K_{\mathfrak{p}}$-isomorphic to $F_{n}%
=F_{n}\!\left(  a,b\right)  $. 

Next, let $\left(  x,y\right)  \longmapsto\left(  u_{n}^{2}x+r_{n},u_{n}%
^{3}y+u_{n}^{2}s_{n}x+w_{n}\right)  $ be an admissible change of variables
resulting in a minimal model $F_{n}^{\prime}$ at $\mathfrak{p}$ for $F_{n}$.
Since $F_{n}$ is given by an integral Weierstrass model, we have by
\cite[Proposition VII.1.3]{MR2514094} that $u_{n},r_{n},s_{n},w_{n}\in
R_{K_{\mathfrak{p}}}$. In particular, the minimal discriminant of $F_{n}$ at
$\mathfrak{p}$ is $\Delta_{n}=u_{n}^{-12}\gamma_{n}$. Similarly, the invariant
$c_{4}$ of $F_{n}^{\prime}$ is given by $c_{4,n}=u_{n}^{-4}\alpha_{n}$. Since $F_n$ has additive reduction at $\mathfrak{p}$, it follows that $v_{\mathfrak{p}}\!\left(  \alpha_{n}\right)  ,v_{\mathfrak{p}}\!\left(
\gamma_{n}\right)  >0$. By
Lemma \ref{GCD}, $\alpha
_{n}R_{K_{\mathfrak{p}}}+\gamma_{n}R_{K_{\mathfrak{p}}}\subseteq
nR_{K_{\mathfrak{p}}}.$ Hence $\mathfrak{p}|n$.
\end{proof}

\section{Isogeny Graphs of Rational Elliptic Curves}\label{section5}
In this section, we prove Theorem~\ref{Thm2}. We begin by considering
those integers $n$ such that $X_{0}\!\left(  n\right)  $ has positive genus
with a non-cuspidal $\mathbb{Q}
$-rational point. That is, $n=11$, $14$, $15$, $17$, $19$, $21$, $27$, $37$, $43$, $67$, $163$. The non-cuspidal $\mathbb{Q}$-rational points of $X_{0}\!\left(  n\right)  $ for $X_{0}\!\left(  n\right)  $ having positive genus are well-understood, and what follows is a summary of the discussion and results that appear in \cite{MR3084348,MR482230}. For these
values of $n$, let $\nu_{n}$ denote the number of $\mathbb{Q}$-rational non-cuspidal points on $X_{0}\!\left(  n\right)  $. In particular, we have that, up to twist, there are $\nu_n$ elliptic curves $E/\mathbb{Q}$ that admit a $\mathbb{Q}$-rational $n$-isogeny. 
By \cite{MR482230},
\[
\nu_{n}=\left\{
\begin{array}
[c]{cl}
1 & \text{if }n=19,27,43,67,163,\\
2 & \text{if }n=14,17,37,\\
3 & \text{if }n=11,\\
4 & \text{if }n=15,21.
\end{array}
\right.
\]
These non-cuspidal points determine, up to quadratic twist, the following isogeny
classes (given by their LMFDB label) of rational elliptic curves:%
\begin{equation}\label{isogen1}
\begin{tabular}
[c]{ccccccc}%
$n$ & $11$ & $14$ & $15$ & $17$ & $19$ & $21$\\\hline
Isogeny class & \href{https://www.lmfdb.org/EllipticCurve/Q/121/a/}{121.a},\href{https://www.lmfdb.org/EllipticCurve/Q/121/b/}{121.b} & \href{https://www.lmfdb.org/EllipticCurve/Q/49/a/}{49.a} & \href{https://www.lmfdb.org/EllipticCurve/Q/50/a/}{50.a} & \href{https://www.lmfdb.org/EllipticCurve/Q/14450/b/}{14450.b} & \href{https://www.lmfdb.org/EllipticCurve/Q/361/a/}{361.a} & \href{https://www.lmfdb.org/EllipticCurve/Q/162/b/}{162.b}\\\hline \hline
$n$ & $27$ & $37$ & $43$ & $67$ & $163$ & \\ \hline
Isogeny class & \href{https://www.lmfdb.org/EllipticCurve/Q/27/a/}{27.a} & \href{https://www.lmfdb.org/EllipticCurve/Q/1225/b/}{1225.b} & \href{https://www.lmfdb.org/EllipticCurve/Q/1849/b/}{1849.b} & \href{https://www.lmfdb.org/EllipticCurve/Q/4489/b/}{4489.b} & \href{https://www.lmfdb.org/EllipticCurve/Q/26569/a/}{26569.a} & \\\hline
\end{tabular}
\end{equation}
Next, let $\mathcal{C}_{n,i}(d)$ be as defined in Table~\ref{ta:curves1}. Then $\mathcal{C}_{n,i}(d)$ is the quadratic twist of $\mathcal{C}_{n,i}(1)$. Moreover, $\mathcal{C}_{n,i}(1)$ is $\mathbb{Q}$-isomorphic to the following elliptic curve (given by their LMFDB label):

{\begingroup
\renewcommand{\arraystretch}{1.2}
\captionof{table}{LMFDB Label of $\mathcal{C}_{n,i}(1)$}
\begin{center}
\begin{tabular}
[c]{ccccc}
\diagbox[width=3em]{$n$}{\vspace{-1.5em}$i$} & $1$ & $2$ & $3$ & $4$\\\hline
$11$ & \href{https://www.lmfdb.org/EllipticCurve/Q/121.a1/}{121.a1}& \href{https://www.lmfdb.org/EllipticCurve/Q/121.a2/}{121.a2} & \href{https://www.lmfdb.org/EllipticCurve/Q/121.b1/}{121.b1} & \href{https://www.lmfdb.org/EllipticCurve/Q/121.b1/}{121.b2}\\\hline
$14$ & \href{https://www.lmfdb.org/EllipticCurve/Q/49.a1/}{49.a1} & \href{https://www.lmfdb.org/EllipticCurve/Q/49.a2/}{49.a2} & \href{https://www.lmfdb.org/EllipticCurve/Q/49.a3/}{49.a3} & \href{https://www.lmfdb.org/EllipticCurve/Q/49.a4/}{49.a4}\\\hline
$15$ & \href{https://www.lmfdb.org/EllipticCurve/Q/50.a1/}{50.a1} & \href{https://www.lmfdb.org/EllipticCurve/Q/50.a3/}{50.a3} & \href{https://www.lmfdb.org/EllipticCurve/Q/50.a4/}{50.a4} & \href{https://www.lmfdb.org/EllipticCurve/Q/50.a2/}{50.a2}\\\hline
$17$ & \href{https://www.lmfdb.org/EllipticCurve/Q/14450.b1/}{14450.b1} & \href{https://www.lmfdb.org/EllipticCurve/Q/14450.b2/}{14450.b2} &  & \\\hline
$19$ & \href{https://www.lmfdb.org/EllipticCurve/Q/361.a1/}{361.a1} & \href{https://www.lmfdb.org/EllipticCurve/Q/361.a2/}{361.a2} &  & \\\hline
$21$ & \href{https://www.lmfdb.org/EllipticCurve/Q/162.b1/}{162.b1} & \href{https://www.lmfdb.org/EllipticCurve/Q/162.b2/}{162.b2} & \href{https://www.lmfdb.org/EllipticCurve/Q/162.b4/}{162.b4}  & \href{https://www.lmfdb.org/EllipticCurve/Q/162.b3/}{162.b3} \\\hline
$27$ & \href{https://www.lmfdb.org/EllipticCurve/Q/27.a1/}{27.a1} & \href{https://www.lmfdb.org/EllipticCurve/Q/27.a3/}{27.a3} & \href{https://www.lmfdb.org/EllipticCurve/Q/27.a4/}{27.a4} & \href{https://www.lmfdb.org/EllipticCurve/Q/27.a2/}{27.a2}\\\hline
$37$ & \href{https://www.lmfdb.org/EllipticCurve/Q/1225.b1/}{1225.b1} & \href{https://www.lmfdb.org/EllipticCurve/Q/1225.b2/}{1225.b2} &  & \\\hline
$43$ & \href{https://www.lmfdb.org/EllipticCurve/Q/1849.b1/}{1849.b1} & \href{https://www.lmfdb.org/EllipticCurve/Q/1849.b2/}{1849.b2} &  & \\\hline
$67$ & \href{https://www.lmfdb.org/EllipticCurve/Q/4489.b1/}{4489.b1} & \href{https://www.lmfdb.org/EllipticCurve/Q/4489.b2/}{4489.b2} &  & \\\hline
$163$ & \href{https://www.lmfdb.org/EllipticCurve/Q/26569.a1/}{26569.a1} & \href{https://www.lmfdb.org/EllipticCurve/Q/26569.a2/}{26569.a2} &  &
\end{tabular}
\end{center}
\endgroup}

We note that $\{[\mathcal{C}_{11,i}(1)]_\mathbb{Q}\}_{i=1}^2$ and $\{[\mathcal{C}_{11,i}(1)]_\mathbb{Q}\}_{i=3}^4$ correspond to the isogeny classes \href{https://www.lmfdb.org/EllipticCurve/Q/121/a/}{121.a} and \href{https://www.lmfdb.org/EllipticCurve/Q/121/b/}{121.b}, respectively. For $n\neq 11$, we have that $\{[\mathcal{C}_{n,i}(1)]_\mathbb{Q}\}_{i}$ corresponds to the isogeny class determined by $n$ in \eqref{isogen1}.
The following result now follows from our discussion:

\begin{lemma}\label{LemmaisographsCM}
Let $n$ be an integer such that $X_{0}\!\left(  n\right)  $ has positive
genus. If $E/\mathbb{Q}$ is an elliptic curve with isogeny class degree $n$, then there exists a squarefree integer $d$ such that the isogeny class of $E$ is 
\[
\begin{array}
[c]{cl}
\{  [  \mathcal{C}_{n,i}(  d)  ]  _{\mathbb{Q}
}\}_{i=1}^{2}  \text{ or }\{  [  \mathcal{C}_{n,i}(  d)
]  _{\mathbb{Q}}\}_{i=3}^{4}   & \text{if }n=11,\\
\{  [  \mathcal{C}_{n,i}(  d)  ]  _{\mathbb{Q}}\}     & \text{otherwise.}
\end{array}
\]
Moreover, the isogeny graph and isogeny matrix of $E$ are as given in Table~\ref{isographsCM}.
\end{lemma}
{\begingroup
\renewcommand{\arraystretch}{1.4}
 \begin{longtable}{ccc}
 	\caption{Isogeny graph and matrix of $\{
[  \mathcal{C}_{n,i}(t,d)  ]_\mathbb{Q} \}  _{i}$}\label{isographsCM}\\
	\hline
	$n$  & Isogeny Graph & Isogeny Matrix \\
	\hline

	\endfirsthead
	\hline
	$n$  & Isogeny Graph & Isogeny Matrix\\
	\hline
	\endhead
	\hline

	\multicolumn{3}{r}{\emph{continued on next page}}
	\endfoot
	\hline
	\endlastfoot

$11,17,19,37,43,67,163$& \begin{tikzcd}
\mathcal{C}_{n,1} \arrow[r, "n", no head] & \mathcal{C}_{n,2}
\end{tikzcd} & \resizebox{.08\textwidth}{!}{$\left(
\begin{array}
[c]{cc}%
1 & n\\
n & 1
\end{array}
\right)  $}	\\\hline

$\begin{array}
[c]{c}%
14,15,21 \\
n=pq,\\
\text{where }1<p<q
\end{array}$ & \begin{tikzcd}
\mathcal{C}_{n,1} \arrow[r, no head, "p"] \arrow[d,no head, "q"] & \mathcal{C}_{n,2} \arrow[d,no head, "q"] \\
\mathcal{C}_{n,3} \arrow[r, no head,"p"]               & \mathcal{C}_{n,4}               
\end{tikzcd} & \resizebox{.15\textwidth}{!}{$\left(
\begin{array}
[c]{cccc}%
1 & p & q & n\\
p & 1 & n & q\\
q & n & 1 & p\\
n & q & p & 1
\end{array}
\right)  $} \\\hline

$27$ & \begin{tikzcd}
\mathcal{C}_{n,1} \arrow[r, no head, "3"] & \mathcal{C}_{n,2} \arrow[r,no head, "3"] & \mathcal{C}_{n,3} \arrow[r,no head, "3"] & \mathcal{C}_{n,4}
\end{tikzcd} &\resizebox{.12\textwidth}{!}{$\left(
\begin{array}
[c]{cccc}%
1 & 3 & 9 & 27\\
3 & 1 & 3 & 9\\
9 & 3 & 1 & 3 \\
27 & 9 & 3 & 1
\end{array}
\right)  $} \\\hline

\end{longtable}
\endgroup}

Our next result considers the case when $\pi:E_{1}\rightarrow E_{2}$ is an
$n$-isogeny of elliptic curves over~$\mathbb{Q}$, such that $j\!\left(  E_{1}\right)  =j\!\left(  E_{2}\right)  \in\left\{
0,1728\right\}  $.

\begin{lemma}\label{j0or1728}
Let $E_{1}$ and $E_{2}$ be elliptic curves defined over $\mathbb{Q}$. Suppose further that $j\!\left(  E_{1}\right)  =j\!\left(  E_{2}\right)
\in\left\{  0,1728\right\}  $ and that $\pi:E_{1}\rightarrow E_{2}$ is a $\mathbb{Q}$-rational $n$-isogeny for $n>1$. Let%
\[
m=\left\{
\begin{array}
[c]{cl}%
6 & \text{if }j\!\left(  E_{1}\right)  =1728,\\
4 & \text{if }j\!\left(  E_{1}\right)  =0.
\end{array}
\right.
\]
Then there is a $d\in\mathbb{Q}^{\times}/\left(\mathbb{Q}^{\times}\right)  ^{m}$ such that $E_{i}$ is $\mathbb{Q}$-isomorphic to $\mathcal{C}_{2,i}^{1728}\!\left(  d\right)  $ or
$\mathcal{C}_{3,i}^{0}\!\left(  d\right)  $ where%
\begin{align*}
\mathcal{C}_{2,1}^{1728}\!\left(  d\right)    & :y^{2}=x^{3}-dx, \qquad & \mathcal{C}_{3,1}^{0}\!\left(  d\right)    & :y^{2}=x^{3}+16d, \\
\mathcal{C}_{2,2}^{1728}\!\left(  d\right)    & :y^{2}=x^{3}+4d, \qquad
& \mathcal{C}_{3,2}^{0}\!\left(  d\right)    & :y^{2}=x^{3}-432d.
\end{align*}

\end{lemma}

\begin{proof}
We begin by observing that
\begin{equation}\label{j01728isodeg}
\left\{  t\in\mathbb{Q}\mid j_{n,1}\!\left(  t\right)  =j_{n,2}\!\left(  t\right)  \in\left\{
0,1728\right\}  \right\}  =\left\{
\begin{array}
[c]{cl}%
\left\{  -64\right\}   & \text{if }n=2,\\
\left\{  -27\right\}   & \text{if }n=3,\\
\varnothing & \text{otherwise.}%
\end{array}
\right.
\end{equation}
In particular, the isogeny $\pi:E_{1}\rightarrow E_{2}$ has degree
$n\in\left\{  2,3\right\}  $. Now observe that $\mathcal{C}_{2,1}%
^{1728}\!\left(  1\right)  $, $\mathcal{C}_{2,2}^{1728}\!\left(  1\right)  $,
$\mathcal{C}_{3,1}^{0}\!\left(  1\right)  $, and $\mathcal{C}_{3,2}%
^{0}\!\left(  1\right)  $ are $\mathbb{Q}$-isomorphic to the following elliptic curves (given by their LMFDB label):
\href{https://www.lmfdb.org/EllipticCurve/Q/32.a3/}{32.a3}, \href{https://www.lmfdb.org/EllipticCurve/Q/32.a4/}{32.a4}, \href{https://www.lmfdb.org/EllipticCurve/Q/27.a4/}{27.a4}, and \href{https://www.lmfdb.org/EllipticCurve/Q/27.a3/}{27.a3}, respectively. In particular, $\mathcal{C}%
_{2,1}^{1728}\!\left(  1\right)  $ (resp. $\mathcal{C}_{3,1}^{0}\!\left(
1\right)  $) and $\mathcal{C}_{2,2}^{1728}\!\left(  1\right)  $ (resp. $\mathcal{C}%
_{3,2}^{0}\!\left(  1\right)  $) are $2$ (resp. $3$)-isogenous. It follows by
\cite[Proposition X.5.4]{MR2514094} that there is a $d\in\mathbb{Q}^{\times}/\left(\mathbb{Q}^{\times}\right)  ^{m}$ such that $E_{i}$ is $\mathbb{Q}$-isomorphic to $\mathcal{C}_{2,i}^{1728}\!\left(  d\right)  $ (resp.
$\mathcal{C}_{3,i}^{0}\!\left(  d\right)  $).
\end{proof}

Before proving Theorem~\ref{Thm2}, it remains to show that for a fixed $n>1$ such that
$X_{0}\!\left(  n\right)  $ has genus $0$, the families of elliptic curves
$\left\{  \mathcal{C}_{n,i}\!\left(  t,1\right)  \right\}  _{i}$ are pairwise
non-isomorphic over $\mathbb{Q}$. This is established by the following lemma:

\begin{lemma}
\label{LemmaQJ}Let $n>1$ be an integer such that $X_{0}\!\left(  n\right)  $ has genus $0$.
If $t,d\in\mathbb{Q}$ and $\mathcal{C}_{n,1}\!\left(  t,d\right)  $ is an elliptic curve, then the
elliptic curves $\mathcal{C}_{n,i}\!\left(  t,d\right)  $ are pairwise
non-isomorphic over $\mathbb{Q}$.
\end{lemma}

\begin{proof}
Let $J_{n,i}\!\left(  t\right)  $ denote the $j$-invariant of $\mathcal{C}%
_{n,i}\!\left(  t,d\right)  $. Set
\[
S_{n,i,k}=\left\{  t\in\mathbb{Q}\mid\mathcal{C}_{n,i}\!\left(  t,1\right)  \text{ is an elliptic curve and
}J_{n,i}\!\left(  t\right)  =J_{n,k}\!\left(  t\right)  \right\}  .
\]
For $i\neq k$, \cite[Verification\_of\_Lemma\_5\_3]{GitHubExplicitIsogenies}
verifies that $S_{n,i,k}=\varnothing$ if $n=5$,$8$,$10$,$12$,$13$,$16$,$18$,$25$. For the
remaining cases, loc. cit. verifies that if $S_{n,i,k}\neq\varnothing$ with
$i\neq k$, then $\mathcal{C}_{n,i}\!\left(  t,1\right)  $ is not $\mathbb{Q}$-isomorphic to $\mathcal{C}_{n,k}\!\left(  t,1\right)  $ for each $t\in
S_{n,i,k}$. It follows that for each $t,d\in\mathbb{Q}$ such that $\mathcal{C}_{n,1}\!\left(  t,d\right)  $ is an elliptic curve,
the elliptic curves $\mathcal{C}_{n,i}\!\left(  t,d\right)  $ are pairwise
non-isomorphic over~$\mathbb{Q}$.
\end{proof}

\begin{proof}[Proof of Theorem~\ref{Thm2}.]
If $E/\mathbb{Q}$ has isogeny class degree $n>1$, then $X_{0}\!\left(  n\right)  $ has a non-cuspidal $\mathbb{Q}$-rational point. If $X_{0}\!\left(  n\right)  $ has positive genus, then the result follows by Lemma~\ref{LemmaisographsCM}. So
suppose $X_{0}\!\left(  n\right)  $ has genus~$0$. Since $E$ has isogeny class
degree $n$, there are $n$-isogenous elliptic curves $E_{1}$ and $E_{2}$
defined over $\mathbb{Q}$ such that $\left[  E_{i}\right]  _{\mathbb{Q}}\in\mathcal{I}_{E/\mathbb{Q}}$.
If the $j$-invariants of $E_{1}$ and $E_{2}$ are not both identically $0$ or
$1728$, then Theorem \ref{CorX0} implies that there is a $t\in\mathbb{Q}$ and $d\in\mathbb{Q}^{\times}/(\mathbb{Q}^{\times})^{2}$ such that the isogeny class of $E_{i}$ contains $\{
[  \mathcal{C}_{n,i}\!\left(  t,d\right)  ]  _{\mathbb{Q}}\}_i  $. By Lemma \ref{LemmaQJ},
\[
\#\left\{  \left[  \mathcal{C}_{n,i}\!\left(  t,d\right)  \right]  _{\mathbb{Q}}\right\}_i  =\left\{
\begin{array}
[c]{cl}%
2 & \text{if }n=2,3,5,7,13,\\
3 & \text{if }n=9,25,\\
4 & \text{if }n=4,6,10,\\
6 & \text{if }n=8,18,\\
8 & \text{if }n=12,16.
\end{array}
\right.
\]
In particular, $\mathcal{I}_{E}=\{  [  \mathcal{C}_{n,i}\!\left(
t,d\right)  ]  _{\mathbb{Q}}\}_i  $ \cite[Theorem 1.2]{MR4203041}.

Now suppose that the $j$-invariant of $E_{1}$ and $E_{2}$ are both identically
$0$ or $1728$. By \eqref{j01728isodeg}, we have that $n$ is either $2$ or $3$. By Lemma~\ref{j0or1728}, we conclude that if $n=2$ (resp. $n=3$), 
then there is a $d \in \mathbb{Q}^{\times}/(\mathbb{Q}^{\times})^{m}$  such that $E_{i}$ is $\mathbb{Q}$-isomorphic $\mathcal{C}_{2,i}^{1728}\!\left(  d\right)  $ (resp. $\mathcal{C}_{3,i}^{0}\!\left(  d\right)  $). In particular, the isogeny class of $E$ is $\mathcal{I}_{E}=\{ [  \mathcal{C}_{2,i}^{1728}\!\left(  d\right)
]  _{\mathbb{Q}}\}_i  $ (resp. $\mathcal{I}
_{E}=\{  [  \mathcal{C}_{3,i}^{0}\!\left(  d\right)  ]  _{\mathbb{Q}}\}_i  $).
\end{proof}

\section{Algorithms for Computing Isogeny Classes}\label{sec:algorithms}

Let $\ell$ be a prime such that $X_{0}\!\left(  \ell\right)  $ has genus $0$.
Suppose further that $K$ is a field of characteristic not equal to $2,3,$ or
$\ell$. For an elliptic curve $E$ defined over $K$, Tsukazaki \cite[Algorithms
2,3,4]{MR3389382} provided algorithms that output elliptic curves that are $\ell$-isogenous
to $E$ over $K$. The algorithm builds on the work of Cremona and Watkins by
computing the kernel polynomial of a $K$-rational $\ell$-isogeny admitted by
$E$. The Weierstrass model of the $\ell$-isogenous elliptic curves are then
obtained via Kohel's formulas \cite{MR2695524}. We note that Algorithms 2, 3,
and 4 of Tsukazaki correspond to the cases when $j\!\left(  E\right)
\neq0,1728$, $j\!\left(  E\right)  =0$, and $j\!\left(  E\right)  =1728$,
respectively. 

In this section, we use Theorem \ref{CorX0} to extend \cite[Algorithm
2]{MR3389382} in the following manner: Let $n>1$ be an integer such that
$X_{0}\!\left(  n\right)  $ has genus $0$, and let $K$ be a field of
characteristic $0$ or relatively prime to $6n$. If $E/K$ is an elliptic curve
with $j\!\left(  E\right)  \neq0,1728$, then Algorithm \ref{Algor} below
outputs a set consisting of elliptic curves that are
$m$-isogenous to $E$ over $K$, where $m$ divides $n$.

To present our algorithm, let $\mathcal{A}_{n,i}\!\left(  t,d\right)  $,
$\mathcal{B}_{n,i}\!\left(  t,d\right)  $, and $\mathcal{C}_{n,i}\!\left(
t,d\right)  $ be as defined in Table \ref{ta:curves}. Set $m_{n}$ to be the
quantity defined below:%
\[%
\begin{tabular}
[c]{cccccc}%
$n$ & $2,3,5,7,13,17,19,37,43,67,163$ & $9,25$ & $4,6,10,11,14,15,21,27$ &
$8,18$ & $12,16$\\\hline
$m_{n}$ & $2$ & $3$ & $4$ & $6$ & $8$%
\end{tabular}
\]
For $n>1$ an integer with $X_{0}\!\left(  n\right)  \!\left(
\mathbb{Q}
\right)  $ having a non-cuspidal point, define
\[
J_{n,i}\!\left(  t\right)  =\left\{
\begin{array}
[c]{cl}%
j\!\left(  \mathcal{C}_{n,i}\!\left(  t,1\right)  \right)   & \text{if }%
X_{0}\!\left(  n\right)  \text{ has genus }0\text{ and }1\leq i\leq m_{n},\\
j\!\left(  \mathcal{C}_{n,i}\!\left(  1\right)  \right)   & \text{if }%
X_{0}\!\left(  n\right)  \text{ has positive genus and }1\leq
i\leq m_{n}.
\end{array}
\right.
\]
Now recall the following result:
\begin{lemma}
[{\cite[Exercise 10.21]{MR2514094}}]\label{Lemmatwist}Let $K$ be a field of
characteristic not equal to $2$ or $3$ and let $E$ and $E^{\prime}$ be
elliptic curves defined over $K$ such that $j\!\left(  E\right)  =j\!\left(
E^{\prime}\right)  \not \in \left\{  0,1728\right\}  $. Let $E:y^{2}%
=x^{3}+Ax+B$ for some $A,B\in K$ and let $c_{4}$ and $c_{6}$ denote the
invariants of the model of $E^{\prime}$. Then $E^{\prime}$ is $K$-isomorphic
to $E^{d}$, where $E^{d}$ is the quadratic twist of $E$ by $d=\frac{Bc_{4}%
}{2Ac_{6}}$.
\end{lemma}
A consequence of this lemma is the following result:
\begin{corollary}
Let $n>1$ be an integer such that $X_{0}\!\left(  n\right)  $ has genus $0$.
Suppose further that $K$ is a field of characteristic $0$ or relatively prime
to $6n$ and that $E/K$ is an elliptic curve with $j\!\left(  E\right)
\neq0,1728$. If $j\!\left(  E\right)  =J_{n,i}\!\left(  t\right)  $ for some
$n,i,t$ with $t \in K$, then $E$ is $K$-isomorphic to $\mathcal{C}_{n,i}\!\left(
t,\frac{2\mathcal{B}_{n,i}\!\left(  t,1\right)  c_{4}}{\mathcal{A}%
_{n,i}\!\left(  t,1\right)  c_{6}}\right)  $, where $c_{4}$ and $c_{6}$ are
the invariants of the Weierstrass model of $E$.
\end{corollary}

\begin{proof}
If $j\!\left(  E\right)  =J_{n,i}\!\left(  t\right)  $ for some $n,i,t$ with $t\in K$, then
$E$ is a twist of $\mathcal{C}_{n,i}\!\left(  t,1\right)  $. Since $j\!\left(
E\right)  \neq0,1728$, $E$ is a quadratic twist of $\mathcal{C}_{n,i}\!\left(
t,1\right)  $. Let $c_{4}$ and $c_{6}$ denote the invariants of $E$. By Lemma
\ref{Lemmatwist}, $E$ is $K$-isomorphic to $\mathcal{C}_{n,i}\!\left(
t,\frac{2\mathcal{B}_{n,i}\!\left(  t,1\right)  c_{4}}{\mathcal{A}%
_{n,i}\!\left(  t,1\right)  c_{6}}\right)  $.
\end{proof}

The following algorithm is now immediate from the above discussion and Theorem
\ref{CorX0}. We note that when $n$ is prime, this is a modified version of Algorithm~2 in \cite{MR3389382}. The main change, is that Kohel's formulas are no longer required to deduce the Weierstrass model of an $n$-isogenous elliptic curve.

\begin{algorithm}[H]    
    \caption{\texttt{isogenies\_genus\_0(E,n)}}\label{Algor}
    \begin{algorithmic}[1]
        \Require{$E$ with $j(E)\not \in \{0,1728\}$, $n$ with $X_0(n)$ having genus $0$ and $n>1$}
        \Ensure{ A set of $K$-rational $m$-isogenous elliptic curves to $E$, with $m|n$ }
        \State Compute $j(E)$ 
        \State Compute $R=\{t \in  K \mid j(E) = j_{n,1}(t) \}$ 
        \State Compute $c_4$, $c_6$ of $E$
        \State Let $T(t) = \frac{\mathcal{B}_{n,k_1}(t,1)c_4  }{2\mathcal{A}_{n,k_1}(t,1)c_6}$
         \State \textbf{return} $\{\mathcal{C}_{n,i}(t,T(t)) \mid t \in R, 1\leq i \leq m_n \}$

    \end{algorithmic}
\end{algorithm}

\begin{example}
Let $K=%
\mathbb{Q}
\!\left(  \sqrt{-3}\right)  $ and let $E:y^{2}=x^{3}-x^{2}+11x-3$ (LMFDB label
\href{https://www.lmfdb.org/EllipticCurve/2.0.3.1/92416.2/u/5}{2.0.3.1-92416.2-u5}). The isogeny class of $E$ over $K$ is $\left\{  \left[
E_{i}\right]  _{K}\mid1\leq i\leq5\right\}  $ where%
\begin{align*}
E_{1}  & :y^{2}=x^{3}-x^{2}-12309x+529757  \hspace{14em} \text{(\href{https://www.lmfdb.org/EllipticCurve/2.0.3.1/92416.2/u/1}{2.0.3.1-92416.2-u1}),}\\
E_{2}  & :y^{2}=x^{3}-x^{2}+\left(  1040\sqrt{-3}-549\right)  x+16976\sqrt
{-3}+10637 \hspace{3.3em}  \text{(\href{https://www.lmfdb.org/EllipticCurve/2.0.3.1/92416.2/u/2}{2.0.3.1-92416.2-u2}),}\\
E_{3}  & :y^{2}=x^{3}-x^{2}+\left(  -1040\sqrt{-3}-549\right)  x-16976\sqrt
{-3}+10637 \hspace{2.5em} \text{(\href{https://www.lmfdb.org/EllipticCurve/2.0.3.1/92416.2/u/3}{2.0.3.1-92416.2-u3}),}\\
E_{4}  & :y^{2}=x^{3}-x^{2}-149x+797 \hspace{16.6em}\text{(\href{https://www.lmfdb.org/EllipticCurve/2.0.3.1/92416.2/u/4}{2.0.3.1-92416.2-u4}),}\\
E_{5}  & :y^{2}=x^{3}-x^{2}+11x-3\hspace{18.1em}\text{(\href{https://www.lmfdb.org/EllipticCurve/2.0.3.1/92416.2/u/5}{2.0.3.1-92416.2-u5}).}
\end{align*}
Next, we apply Algorithm \ref{Algor} to $E$ with $n=9$. With notation as in
Algorithm \ref{Algor}, we observe that $R=\left\{  t_{1},t_{2},t_{3}\right\}
$ where $t_{1}=\frac{27\sqrt{-3}-51}{19},$ $t_{2}=\frac{-27\sqrt{-3}-51}{19}$,
and $t_{3}=-24$. The invariants $c_{4}$ and $c_{6}$ of $E$ are $c_{4}%
=c_{6}=-512$. Now set $d_{i}=T\!\left(  t_{i}\right)  $ for $i=1,2,3$ where
$d_{1}=\frac{1701\sqrt{-3}-891}{2888}$, $d_{2}=\frac{-1701\sqrt{-3}-891}%
{2888}$, and $d_{3}=\frac{-81}{4}$. Algorithm \ref{Algor} returns $\left\{
\mathcal{C}_{9,i}\!\left(  t_{i},d_{i}\right)  \mid1\leq i\leq3\right\}  $.
We note that%
\begin{align*}
\left[  E_{1}\right]  _{K}  & =\left[  \mathcal{C}_{9,3}\!\left(  t_{3}%
,d_{3}\right)  \right]  _{K},\\
\left[  E_{2}\right]  _{K}  & =\left[  \mathcal{C}_{9,3}\!\left(  t_{1}%
,d_{1}\right)  \right]  _{K},\\
\left[  E_{3}\right]  _{K}  & =\left[  \mathcal{C}_{9,3}\!\left(  t_{2}%
,d_{2}\right)  \right]  _{K},\\
\left[  E_{4}\right]  _{K}  & =\left[  \mathcal{C}_{9,2}\!\left(  t_{1}%
,d_{1}\right)  \right]  _{K}=\left[  \mathcal{C}_{9,2}\!\left(  t_{2}%
,d_{2}\right)  \right]  _{K}=\left[  \mathcal{C}_{9,2}\!\left(  t_{3}%
,d_{3}\right)  \right]  _{K},\\
\left[  E_{5}\right]  _{K}  & =\left[  \mathcal{C}_{9,1}\!\left(  t_{1}%
,d_{1}\right)  \right]  _{K}=\left[  \mathcal{C}_{9,1}\!\left(  t_{2}%
,d_{2}\right)  \right]  _{K}=\left[  \mathcal{C}_{9,1}\!\left(  t_{3}%
,d_{3}\right)  \right]  _{K}.
\end{align*}
In particular, the isogeny class $\mathcal{I}_{E/K}$ contains the set%
\[
\left\{  \left[  \mathcal{C}_{9,1}\!\left(  t_{1},d_{1}\right)  \right]
_{K},\left[  \mathcal{C}_{9,2}\!\left(  t_{1},d_{1}\right)  \right]
_{K},\left[  \mathcal{C}_{9,3}\!\left(  t_{1},d_{1}\right)  \right]
_{K},\left[  \mathcal{C}_{9,3}\!\left(  t_{2},d_{2}\right)  \right]
_{K},\left[  \mathcal{C}_{9,3}\!\left(  t_{3},d_{3}\right)  \right]
_{K}\right\}  .
\]
By Theorem \ref{CorX0} and the above, we have that the isogeny graph of $E/K$
contains the following subgraph:
\[
\begin{tikzcd}
                                        & E_1 \arrow[d, "3" description, no head]                                     &     \\
E_5 \arrow[r, "3" description, no head] & E_4 \arrow[r, "3" description, no head] \arrow[d, "3" description, no head] & E_2 \\
                                        & E_3                                                                         &    
\end{tikzcd}
\]

\noindent Lastly, we note that by LMFDB, the above isogeny graph is the isogeny graph of
$E$.
\end{example}

Next, let $J=\left\{  J_{n,i}\!\left(  t\right)  \mid X_{0}\!\left(  n\right)\text{ has positive genus and }1\leq i\leq m_{n}\right\}  \cup\left\{1728\right\}  $. When $K=\mathbb{Q}$, Algorithm \ref{Algor} can be modified using Theorem \ref{Thm2} to yield the isogeny class of a rational elliptic curve $E$ with $j\!\left(  E\right)\not \in J$. Specifically, we have:

\begin{algorithm}[H]    
    \caption{\texttt{isogeny\_class\_of\_E/$\mathbb{Q}$}}\label{Algor2}
    \begin{algorithmic}[1]
        \Require{$E$ with $j(E)\not \in J$}
        \Ensure{ Isogeny class of $E/ \mathbb{Q}$ }
        \State Compute $j(E)$ 
        \State Compute $R=\{(n,i,t) \in \mathbb{N} \times \mathbb{N} \times \mathbb{Q} \mid j(E) = J_{n,i}(t),1\leq i \leq m_n \}$ 
        \State Compute $S=\{ (n,i,t) \in R \mid n \geq n'$ for all $(n',i',t')\in R \}$
        \State Fix $(n,i,t) \in S$
        \State Compute $c_4$, $c_6$ of $E$
        \State Compute $T(n,i,t) = \frac{\mathcal{B}_{n,i}(t,1)c_4  }{2\mathcal{A}_{n,i}(t,1)c_6}$
        \State \textbf{return} $\{\mathcal{C}_{n,j}(t,T(n,i,t)) \mid 1\leq j \leq n_i \}$
    \end{algorithmic}
\end{algorithm}
\noindent We now demonstrate Algorithm~\ref{Algor2} by applying it to the elliptic curve $E:y^2=x^3-15x-22$ (LMFDB \cite{lmfdb} label \href{https://www.lmfdb.org/EllipticCurve/Q/144/a/2}{144.a2}). With notation as in the Algorithm, we observe that:
\[
S=\left\{  \left(  6,1,-6\right)  ,\left(
6,2,-12\right)  ,\left(  6,3,-6\right)  ,\left(  6,4,-12\right)  \right\}  .
\]
Now choose $\left(  6,1,-6\right)  \in S$ and observe that $T\!\left(
6,1,-6\right)  =-1$. Thus, $E$ is $\mathbb{Q}$-isomorphic to $\mathcal{C}_{6,1}\!\left(  -6,-1\right)  $ and we have that
the isogeny class of $E$ is $\{  [  \mathcal{C}_{6,i}\!\left(
-6,-1\right)  ]_{\mathbb{Q}}  \}  _{i=1}^{4}$ where the models of $\mathcal{C}_{6,i}(-6,-1)$ are given by:
\begin{align*}
\mathcal{C}_{6,1}\!\left(  -6,-1\right)    & :y^{2}=x^{3}- 19440x + 1026432,\\
\mathcal{C}_{6,2}\!\left(  -6,-1\right)    & :y^{2}=x^{3}+ 2985984,\\
\mathcal{C}_{6,3}\!\left(  -6,-1\right)    & :y^{2}=x^{3}- 174960x - 27713664,\\
\mathcal{C}_{6,4}\!\left(  -6,-1\right)    & :y^{2}=x^{3}-80621568.
\end{align*}
The isogeny graph of $E$ can then be deduced from Table~\ref{isographs}.

\vspace{1em}

\noindent \textbf{Acknowledgments.} The author would like to thank Imin Chen, Alyson Deines, Edray Goins, Enrique Gonz\'{a}lez-Jim\'{e}nez, Manami Roy, and Andrew V. Sutherland for their helpful comments on this work. The author also thanks the referee for their helpful comments and suggestions.

\begin{appendix}
\section{}
\label{AppendixTables}
\vspace{-0.5em}
{\renewcommand*{\arraystretch}{1.2} \begin{longtable}{ccC{2.15in}C{3.1in}}
\caption{The elliptic curve $\mathcal{C}_{n,i}(t,d):y^2=x^3+d^2\mathcal{A}_{n,i}x+d^3\mathcal{B}_{n,i}$}\\
\hline
$n$ & $i$ & $\mathcal{A}_{n,i}$ & $\mathcal{B}_{n,i}$\\
\hline
\endfirsthead
\caption[]{\emph{continued}}\\
\hline
$n$ & $i$ & $\mathcal{A}_{n,i}$ & $\mathcal{B}_{n,i}$ \\
\hline
\endhead
\hline
\multicolumn{4}{r}{\emph{continued on next page}}
\endfoot
\hline
\endlastfoot

$2$ & $1$ & $ -27  (t + 64)  (t + 256)  $ & $-54  (t - 512)  (t + 64)^{2} $ \\\cmidrule{2-4}
& $2$ & $-432  (t + 16)  (t + 64)   $ & $-3456  (t - 8)  (t + 64)^{2} $ \\\hline

$3$ & $1$ & $-3  (t + 243)    (t + 27)^{3} $ & $-2  (t + 27)^{4}  (t^{2} - 486 t - 19683) $   \\\cmidrule{2-4}
& $2$ & $-243  (t + 3)    (t + 27)^{3} $ & $-1458  (t + 27)^{4}  (t^{2} + 18 t - 27) $ \\\hline

$4$ & $1$ & $-432  (t^{2} + 16 t + 256) $  & $3456  (t - 16)  (t + 8)  (t + 32) $\\\cmidrule{2-4}
& $2$ & $ -6912  (t^{2} + 16 t + 16)$  & $ 221184  (t + 8)  (t^{2} + 16 t - 8)$  \\\cmidrule{2-4}
& $3$ &  $-432  (t^{2} - 224 t + 256) $  &  $3456  (t - 16)  (t^{2} + 544 t + 256) $\\\cmidrule{2-4}
& $4$ &  $-27  (t^{2} + 256 t + 4096) $ &  $54  (t + 32)  (t^{2} - 512 t - 8192) $ \\\hline

$5$ & $1$ & $-27    (t^{2} + 22 t + 125)  (t^{2} + 250 t + 3125) $ & $-54  (t^{2} - 500 t - 15625)  (t^{2} + 22 t + 125)^{2} $\\\cmidrule{2-4}
& $2$ & $-16875    (t^{2} + 10 t + 5)  (t^{2} + 22 t + 125) $ & $-843750  (t^{2} + 4 t - 1)  (t^{2} + 22 t + 125)^{2} $\\\hline

$6$ & $1$ & $-3  (t + 12)  (t^{3} + 252 t^{2} + 3888 t + 15552) $  & $-2  (t^{2} + 36 t + 216)  (t^{4} - 504 t^{3} - 13824 t^{2} - 124416 t - 373248) $\\\cmidrule{2-4}
& $2$ & $-48  (t + 6)  (t^{3} + 18 t^{2} + 324 t + 1944) $ & $-128  (t^{2} + 36 t + 216)  (t^{4} - 216 t^{2} - 1944 t - 5832) $ \\\cmidrule{2-4}
& $3$ & $-243  (t + 12)  (t^{3} + 12 t^{2} + 48 t + 192) $ & $ -1458  (t^{2} + 12 t + 24)  (t^{4} + 24 t^{3} + 192 t^{2} - 4608)$ \\\cmidrule{2-4}
& $4$ & $-3888  (t + 6)  (t^{3} + 18 t^{2} + 84 t + 24) $ & $-93312  (t^{2} + 12 t + 24)  (t^{4} + 24 t^{3} + 192 t^{2} + 504 t - 72) $ \\\hline

$7$ & $1$ & $-27    (t^{2} + 13 t + 49)  (t^{2} + 245 t + 2401) $ & $-54  (t^{2} + 13 t + 49)  (t^{4} - 490 t^{3} - 21609 t^{2} - 235298 t - 823543) $\\\cmidrule{2-4}
& $2$ & $-64827    (t^{2} + 5 t + 1)  (t^{2} + 13 t + 49) $ & $-6353046  (t^{2} + 13 t + 49)  (t^{4} + 14 t^{3} + 63 t^{2} + 70 t - 7) $ \\\hline

$8$ & $1$ & $-432  (t^{4} + 224 t^{2} + 256) $  & $ -3456  (t^{2} - 24 t + 16)  (t^{2} + 16)  (t^{2} + 24 t + 16)$\\\cmidrule{2-4}
& $2$ & $-432  (t^{4} - 240 t^{3} + 2144 t^{2} - 3840 t + 256) $ & $-3456  (t^{2} - 24 t + 16)  (t^{4} + 528 t^{3} - 4000 t^{2} + 8448 t + 256) $ \\\cmidrule{2-4}
& $3$ & $-27  (t^{4} + 240 t^{3} + 2144 t^{2} + 3840 t + 256) $ & $-54  (t^{2} + 24 t + 16)  (t^{4} - 528 t^{3} - 4000 t^{2} - 8448 t + 256) $  \\\cmidrule{2-4}
& $4$ & $-6912  (t^{4} - 16 t^{2} + 256) $ & $ -221184  (t^{2} - 32)  (t^{2} - 8)  (t^{2} + 16)$ \\\cmidrule{2-4}
& $5$ & $-6912  (t^{4} - 256 t^{2} + 4096) $ & $-221184  (t^{2} - 32)  (t^{4} + 512 t^{2} - 8192) $ \\\cmidrule{2-4}
& $6$ & $-110592  (t^{4} - 16 t^{2} + 16) $ & $-14155776  (t^{2} - 8)  (t^{4} - 16 t^{2} - 8) $ \\\hline

$9$ & $1$ & $-3  (t + 6)  (t^{3} + 234 t^{2} + 756 t + 2160) $ & $-2  (t^{6} - 504 t^{5} - 16632 t^{4} - 123012 t^{3} - 517104 t^{2} - 1143072 t - 1475496) $\\\cmidrule{2-4}
& $2$ & $ -243  t  (t + 6)  (t^{2} - 6 t + 36)$  & $-1458  (t^{2} - 6 t - 18)  (t^{4} + 6 t^{3} + 54 t^{2} - 108 t + 324) $ \\\cmidrule{2-4}
& $3$ & $-19683  t  (t^{3} - 24) $ & $ -1062882  (t^{6} - 36 t^{3} + 216)$  \\\hline

$10$ & $1$ & $-27  (t^{2} + 4)  (t^{6} + 236 t^{5} + 1440 t^{4} + 1920 t^{3} + 3840 t^{2} + 256 t + 256) $ & $-54  (t^{2} + 4 t + 8)  (t^{2} + 22 t - 4)  (t^{2} + 4)^{2}  (t^{4} - 536 t^{3} - 264 t^{2} - 416 t - 64) $ \\\cmidrule{2-4}
& $2$ & $-432  (t^{2} + 4)  (t^{6} - 4 t^{5} + 240 t^{4} - 480 t^{3} + 1440 t^{2} - 944 t + 16) $ & $-3456  (t^{2} - 2 t + 2)  (t^{2} + 22 t - 4)  (t^{2} + 4)^{2}  (t^{4} - 26 t^{3} + 66 t^{2} - 536 t - 4) $ \\\cmidrule{2-4}
& $3$ & $ -768  (t^{2} - 3)  (t^{6} - 9 t^{4} + 243 t^{2} - 243)$ & $-8192  (t^{4} + 18 t^{2} - 27)  (t^{8} - 36 t^{6} + 270 t^{4} - 972 t^{2} + 729) $ \\\cmidrule{2-4}
& $4$ & $-62208  (t^{2} - 3)  (t^{6} - 9 t^{4} + 3 t^{2} - 3) $ & $-5971968  (t^{4} - 6 t^{2} - 3)  (t^{8} - 12 t^{6} + 30 t^{4} - 36 t^{2} + 9) $\\\hline

$12$ & $1$ & $-48  (t^{2} + 3)  (t^{6} + 225 t^{4} - 405 t^{2} + 243) $& $-128  (t^{4} - 24 t^{3} + 18 t^{2} - 27)  (t^{4} + 18 t^{2} - 27)  (t^{4} + 24 t^{3} + 18 t^{2} - 27) $ \\\cmidrule{2-4}
& $2$ & $-3888  (t^{2} + 3)  (t^{6} - 15 t^{4} + 75 t^{2} + 3) $ & $-93312  (t^{4} - 6 t^{2} - 24 t - 3)  (t^{4} - 6 t^{2} - 3)  (t^{4} - 6 t^{2} + 24 t - 3) $\\\cmidrule{2-4}
& $3$ & $-768  (t^{2} - 3)  (t^{6} - 9 t^{4} + 243 t^{2} - 243) $  & $-8192  (t^{4} + 18 t^{2} - 27)  (t^{8} - 36 t^{6} + 270 t^{4} - 972 t^{2} + 729) $ \\\cmidrule{2-4}
& $4$ & $-62208  (t^{2} - 3)  (t^{6} - 9 t^{4} + 3 t^{2} - 3) $ & $-5971968  (t^{4} - 6 t^{2} - 3)  (t^{8} - 12 t^{6} + 30 t^{4} - 36 t^{2} + 9) $ \\\cmidrule{2-4}
& $5$ & $-3  (t^{2} + 6 t - 3)  (t^{6} + 234 t^{5} + 747 t^{4} + 540 t^{3} - 729 t^{2} - 486 t - 243) $ & $ -2  (t^{4} + 24 t^{3} + 18 t^{2} - 27)  (t^{8} - 528 t^{7} - 3996 t^{6} - 9504 t^{5} + 270 t^{4} + 14256 t^{3} - 972 t^{2} + 729)$\\\cmidrule{2-4}
& $6$ & $-243  (t^{2} + 6 t - 3)  (t^{6} - 6 t^{5} + 27 t^{4} + 60 t^{3} - 249 t^{2} + 234 t - 3) $ & $ -1458  (t^{4} - 6 t^{2} + 24 t - 3)  (t^{8} - 12 t^{6} - 528 t^{5} + 30 t^{4} + 3168 t^{3} - 3996 t^{2} + 1584 t + 9)$ \\\cmidrule{2-4}
& $7$ & $-48  (t^{2} - 6 t - 3)  (t^{6} - 234 t^{5} + 747 t^{4} - 540 t^{3} - 729 t^{2} + 486 t - 243) $ & $-128  (t^{4} - 24 t^{3} + 18 t^{2} - 27)  (t^{8} + 528 t^{7} - 3996 t^{6} + 9504 t^{5} + 270 t^{4} - 14256 t^{3} - 972 t^{2} + 729) $ \\\cmidrule{2-4}
& $8$ & $-3888  (t^{2} - 6 t - 3)  (t^{6} + 6 t^{5} + 27 t^{4} - 60 t^{3} - 249 t^{2} - 234 t - 3) $ & $-93312  (t^{4} - 6 t^{2} - 24 t - 3)  (t^{8} - 12 t^{6} + 528 t^{5} + 30 t^{4} - 3168 t^{3} - 3996 t^{2} - 1584 t + 9) $ \\\hline

$13$ & $1$ & $-27    (t^{2} + 5 t + 13)  (t^{2} + 6 t + 13)  (t^{4} + 247 t^{3} + 3380 t^{2} + 15379 t + 28561) $ & $-54  (t^{2} + 5 t + 13)  (t^{2} + 6 t + 13)^{2}  (t^{6} - 494 t^{5} - 20618 t^{4} - 237276 t^{3} - 1313806 t^{2} - 3712930 t - 4826809) $ \\\cmidrule{2-4}
& $2$ & $-771147    (t^{2} + 5 t + 13)  (t^{2} + 6 t + 13)  (t^{4} + 7 t^{3} + 20 t^{2} + 19 t + 1) $ & $-260647686  (t^{2} + 5 t + 13)  (t^{2} + 6 t + 13)^{2}  (t^{6} + 10 t^{5} + 46 t^{4} + 108 t^{3} + 122 t^{2} + 38 t - 1) $ \\\hline

$16$ & $1$ & $ -432  (t^{8} + 240 t^{6} + 2144 t^{4} + 3840 t^{2} + 256)$ & $ -3456  (t^{4} - 24 t^{3} + 24 t^{2} - 96 t + 16)  (t^{4} + 24 t^{2} + 16)  (t^{4} + 24 t^{3} + 24 t^{2} + 96 t + 16)$\\\cmidrule{2-4}
& $2$ & $-27  (t^{8} + 240 t^{7} + 2160 t^{6} + 6720 t^{5} + 17504 t^{4} + 26880 t^{3} + 34560 t^{2} + 15360 t + 256) $  & $-54  (t^{4} + 24 t^{3} + 24 t^{2} + 96 t + 16)  (t^{8} - 528 t^{7} - 3984 t^{6} - 14784 t^{5} - 31648 t^{4} - 59136 t^{3} - 63744 t^{2} - 33792 t + 256) $ \\\cmidrule{2-4}
& $3$ & $-432  (t^{8} - 240 t^{7} + 2160 t^{6} - 6720 t^{5} + 17504 t^{4} - 26880 t^{3} + 34560 t^{2} - 15360 t + 256) $ & $-3456  (t^{4} - 24 t^{3} + 24 t^{2} - 96 t + 16)  (t^{8} + 528 t^{7} - 3984 t^{6} + 14784 t^{5} - 31648 t^{4} + 59136 t^{3} - 63744 t^{2} + 33792 t + 256) $ \\\cmidrule{2-4}
& $4$ & $-6912  (t^{4} - 4 t^{3} + 8 t^{2} + 16 t + 16)  (t^{4} + 4 t^{3} + 8 t^{2} - 16 t + 16) $ & $-221184  (t^{2} - 4 t - 4)  (t^{2} + 4 t - 4)  (t^{4} + 16)  (t^{4} + 24 t^{2} + 16) $ \\\cmidrule{2-4}
& $5$ & $-6912  (t^{4} - 16 t^{3} + 8 t^{2} + 64 t + 16)  (t^{4} + 16 t^{3} + 8 t^{2} - 64 t + 16) $ & $-221184  (t^{2} - 4 t - 4)  (t^{2} + 4 t - 4)  (t^{8} + 528 t^{6} - 4000 t^{4} + 8448 t^{2} + 256) $ \\\cmidrule{2-4}
& $6$ & $-110592  (t^{8} - 16 t^{4} + 256) $ & $-14155776  (t^{4} - 32)  (t^{4} - 8)  (t^{4} + 16) $\\\cmidrule{2-4}
& $7$ & $-110592  (t^{8} - 256 t^{4} + 4096) $  & $ -14155776  (t^{4} - 32)  (t^{8} + 512 t^{4} - 8192)$ \\\cmidrule{2-4}
& $8$ & $-1769472  (t^{8} - 16 t^{4} + 16) $  & $-905969664  (t^{4} - 8)  (t^{8} - 16 t^{4} - 8) $ \\\hline

$18$ & $1$ & $-3  (t^{3} + 6 t^{2} + 4)  (t^{9} + 234 t^{8} + 756 t^{7} + 2172 t^{6} + 1872 t^{5} + 3024 t^{4} + 48 t^{3} + 3744 t^{2} + 64) $ & $-2  (t^{6} + 24 t^{5} + 24 t^{4} + 92 t^{3} - 48 t^{2} + 96 t - 8)  (t^{12} - 528 t^{11} - 3984 t^{10} - 14792 t^{9} - 27936 t^{8} - 42624 t^{7} - 37632 t^{6} - 52992 t^{5} - 25344 t^{4} - 43520 t^{3} - 6144 t^{2} - 6144 t - 512) $ \\\cmidrule{2-4}
& $2$ & $-48  (t^{3} + 6 t - 2)  (t^{9} + 234 t^{7} - 6 t^{6} + 756 t^{5} - 936 t^{4} + 2172 t^{3} - 1512 t^{2} + 936 t - 8) $ & $-128  (t^{6} + 24 t^{5} + 24 t^{4} + 92 t^{3} - 48 t^{2} + 96 t - 8)  (t^{12} - 24 t^{11} + 48 t^{10} - 680 t^{9} + 792 t^{8} - 3312 t^{7} + 4704 t^{6} - 10656 t^{5} + 13968 t^{4} - 14792 t^{3} + 7968 t^{2} - 2112 t - 8) $ \\\cmidrule{2-4}
& $3$ & $-243  (t^{3} + 4)  (t^{3} + 6 t^{2} + 4)  (t^{6} - 6 t^{5} + 36 t^{4} + 8 t^{3} - 24 t^{2} + 16) $ & $-1458  (t^{2} + 2 t - 2)  (t^{4} - 8 t^{3} - 8 t - 8)  (t^{4} - 2 t^{3} + 6 t^{2} + 4 t + 4)  (t^{8} + 8 t^{7} + 64 t^{6} - 16 t^{5} - 56 t^{4} + 128 t^{3} + 64 t^{2} - 64 t + 64) $  \\\cmidrule{2-4}
& $4$ & $-3888  (t^{3} - 2)  (t^{3} + 6 t - 2)  (t^{6} - 6 t^{4} - 4 t^{3} + 36 t^{2} + 12 t + 4) $ & $-93312  (t^{2} + 2 t - 2)  (t^{4} - 2 t^{3} - 8 t - 2)  (t^{4} - 2 t^{3} + 6 t^{2} + 4 t + 4)  (t^{8} + 2 t^{7} + 4 t^{6} - 16 t^{5} - 14 t^{4} + 8 t^{3} + 64 t^{2} - 16 t + 4) $ \\\cmidrule{2-4}
& $5$ & $-19683  (t^{3} + 4)  (t^{9} - 12 t^{6} + 48 t^{3} + 64) $  & $-1062882  (t^{6} - 4 t^{3} - 8)  (t^{12} - 8 t^{9} - 512 t^{3} - 512) $ \\\cmidrule{2-4}
& $6$ & $-314928  (t^{3} - 2)  (t^{9} - 6 t^{6} - 12 t^{3} - 8) $ & $-68024448  (t^{6} - 4 t^{3} - 8)  (t^{12} - 8 t^{9} - 8 t^{3} - 8) $\\\hline

$25$ & $1$ & $-27  (t^{2} + 4)  (t^{10} + 240 t^{9} + 2170 t^{8} + 8880 t^{7} + 34835 t^{6} + 83748 t^{5} + 206210 t^{4} + 313380 t^{3} + 503545 t^{2} + 424740 t + 375376) $ & $-54  (t^{2} + 4)^{2}  (t^{4} + 6 t^{3} + 21 t^{2} + 36 t + 61)  (t^{10} - 510 t^{9} - 13580 t^{8} - 36870 t^{7} - 190915 t^{6} - 393252 t^{5} - 1068040 t^{4} - 1508370 t^{3} - 2581955 t^{2} - 2087010 t - 1885124) $ \\\cmidrule{2-4}
& $2$ & $-16875  (t^{2} + 4)  (t^{2} + 3 t + 1)  (t^{4} - 4 t^{3} + 11 t^{2} - 14 t + 31)  (t^{4} + t^{3} + 11 t^{2} - 4 t + 16) $ & $-843750  (t^{2} - 2 t - 4)  (t^{2} + 4)^{2}  (t^{4} - 4 t^{3} + 21 t^{2} - 34 t + 41)  (t^{4} + 3 t^{2} + 1)  (t^{4} + 6 t^{3} + 21 t^{2} + 36 t + 61) $ \\\cmidrule{2-4}
& $3$ &  $-10546875  (t^{2} + 4)  (t^{10} + 10 t^{8} + 35 t^{6} - 12 t^{5} + 50 t^{4} - 60 t^{3} + 25 t^{2} - 60 t + 16) $ & $ -13183593750  (t^{2} + 4)^{2}  (t^{4} + 3 t^{2} + 1)  (t^{10} + 10 t^{8} + 35 t^{6} - 18 t^{5} + 50 t^{4} - 90 t^{3} + 25 t^{2} - 90 t + 76)$
\label{ta:curves}	
\end{longtable}}

{\renewcommand*{\arraystretch}{1.2} \begin{longtable}{cC{1.3in}C{1.5in}C{1.1in}C{1.1in}}
\caption{The elliptic curve $\mathcal{C}_{n,i}(d):y^2=x^3+d^2\mathcal{A}_{n,i}x+d^3\mathcal{B}_{n,i}$}\\
\hline
$n$ &  $\mathcal{A}_{n,1}$ & $\mathcal{B}_{n,1}$ & $\mathcal{A}_{n,2}$ & $\mathcal{B}_{n,2}$ \\
\hline
\endfirsthead
\caption[]{\emph{continued}}\\
\hline
\endhead
\hline
\multicolumn{4}{r}{\emph{continued on next page}}
\endfoot
\hline
\endlastfoot

$11$ & $-1149984 $ & $-487018224 $ & $-9504 $ & $365904$\\\hline

$14$ & $-2361555 $ & $1396762542 $ & $-138915 $ & $24504606 $ \\\hline

$15$ & $-162675 $ & $-25254450 $ & $-675 $  & $-79650 $  \\\hline

$17$  & $-247394115 $  & $-1679010134850 $ & $-3940515 $  & $3010787550 $ \\\hline

$19$ & $-219488 $  & $-39617584  $ & $-608 $  & $5776 $ \\\hline

$21$ & $-1396035$ & $634881726 $ &$-1104435 $ & $907504398 $  \\\hline

$27$ & $-4320 $  & $-109296 $ & $0$ & $-432 $  \\\hline

$37$ & $-269675595 $ & $-1704553285050 $ & $-10395 $ & $444150 $ \\\hline

$43$ & $-25442240 $& $-49394836848 $ & $-13760 $ & $621264 $ \\\hline

$67$ & $-529342880 $ & $-4687634371504 $ & $-117920 $ & $15585808 $ \\\hline

$163$ & $-924354639680 $  & $-342062961763303088 $ & $-34790720 $  & $78984748304 $
\\\hline \hline

$n$ &  $\mathcal{A}_{n,3}$ & $\mathcal{B}_{n,3}$ & $\mathcal{A}_{n,4}$ & $\mathcal{B}_{n,4}$\\\hline
 
 $11$ & $-395307$  & $373960422$  & $-38907$  & $-2953962$ \\\hline
 $14$ & $-48195 $  & $-4072194 $  & $-2835 $ & $-71442 $  \\\hline
 $15$ & $712125 $  & $-104861250 $ &  $-97875 $  & $14208750 $ \\\hline
 $21$ & $3645 $ & $-13122 $ & $-54675 $  & $-5156946 $\\\hline
 $27$ & $0$ & $16 $ & $-480 $ & $4048 $
\label{ta:curves1}	
\end{longtable}}

\end{appendix}
\newpage
\bibliographystyle{plain}
\bibliography{bibliography}

\end{document}